\documentclass[12pt]{amsart}

\pdfoutput=1

\usepackage[utf8]{inputenc}
\usepackage{multirow}
\usepackage{blkarray}
\usepackage{verbatim}
\usepackage{mathtools,graphicx}
\usepackage{kbordermatrix}
\renewcommand{\kbrowstyle}{\relax}
\renewcommand{\kbcolstyle}{\relax}
\usepackage{amsmath,amssymb,enumerate,mathtools,amsthm}
\newtheorem{theorem}{Theorem}[section]
\newtheorem{claim}{Claim}[theorem]
\newtheorem{lemma}[theorem]{Lemma}

\newtheorem{corollary}[theorem]{Corollary}

\newtheorem{hypothesis}[theorem]{Hypothesis}

\newcommand{\bF}{\mathbb F}
\newcommand{\bR}{\mathbb R}

\newcommand{\bT}{\mathbf T}

\newcommand{\cG}{\mathcal{G}}

\newcommand{\cM}{\mathcal{M}}

\newcommand{\nni}{\mathbb{N}_0}
\DeclareMathOperator{\si}{si}

\DeclareMathOperator{\PG}{PG}
\DeclareMathOperator{\DG}{DG}

\DeclareMathOperator{\GF}{GF}
\DeclareMathOperator{\AG}{AG}
\DeclareMathOperator{\col}{col}
\DeclareMathOperator{\row}{row}
\DeclareMathOperator{\rank}{rank}
\newcommand{\elem}{\varepsilon}
\newcommand{\del}{\!\setminus\!}
\newcommand{\con}{/}
\newcommand{\wh}{\widehat}

\newcommand{\templatecrap}{(\Gamma,C,X,Y_0,Y_1,A_1,\Delta,\Lambda)}

\newcommand{\tcol}[1]{\multicolumn{1}{c}{#1}}
\newcommand{\ol}[1]{\overline{#1}}
\newcommand{\sqbinom}[2]{{{#1} \brack {#2}}}
\newcommand{\smallsq}{\scalebox{0.4}{$\square$}}

\DeclareMathOperator{\rowspace}{row}

\newcommand{\sqpinchg}[2]{\cG(#1)^{#2}_{\smallsq}}
\newcommand{\sqpinch}[1]{\cG(\Gamma)^{#1}_{\smallsq}}

\newcommand{\xpinch}[1]{\cG(\Gamma)^{#1}_{(x)}}
\newcommand{\one}[1]{\mathbf{1}_{#1}}
\title[Excluding Geometries]{The extremal function for geometry minors of matroids over prime fields}
\author{Peter Nelson}
\author{Zach Walsh}
\date{November 2021}

\begin{document}

\begin{abstract}
	A frame template over a field $\bF$ describes the precise way in which a given $\bF$-representable matroid is close to being a frame matroid. 
	Our main result determines the maximum-rank projective or affine geometry that is described by a given frame template over a prime field.
	Subject to the matroid minors hypothesis of Geelen, Gerards, and Whittle, we use our result to determine, for each projective or affine geometry $N$ over a prime field $\bF$, a best-possible upper bound on the number of elements in a simple $\bF$-representable matroid $M$ of sufficiently large rank with no $N$-minor.
\end{abstract}

\maketitle

\section{Introduction}
Frame templates were introduced by Geelen, Gerards, and Whittle [\ref{ggwstructure}] as a way to describe the specific structure of classes of matroids representable over a finite field $\bF$.
A frame template describes the precise way in which a given $\bF$-representable matroid is close to being a frame matroid.
For this reason, the definition of frame template is quite technical, and we defer its statement to Section 4.

Our main result, Theorem \ref{templatetech}, states that every frame template over a prime field describes a class of matroids that either contains a large projective geometry or affine geometry, or is degenerate.
Before proving this result, we present a simpler definition of a frame template, and prove in Lemma \ref{equivalent} that it is equivalent to the definition in [\ref{ggwstructure}].
Our definition from an earlier version of this paper has since been further simplified in [\ref{gvzcounterexample}], and applied in [\ref{cgovz}] and [\ref{grace}] to study classes of ternary and quaternary matroids, respectively.

Our motivation for studying frame templates is the matroid minors hypothesis of Geelen, Gerards, and Whittle [\ref{ggwstructure}], whose proof has been announced but not yet written.
Subject to this hypothesis (Hypothesis \ref{structure1}), we show that our main result implies the following theorem.

\begin{theorem}\label{mainsimple}
    Let $t \in \nni$. If $M$ is a simple binary matroid of sufficiently large rank with no $\PG(t+2,2)$-minor, then \[|M| \le 2^t\tbinom{r(M)-t+1}{2} + 2^t-1.\] 
\end{theorem}

We also show that this bound is best-possible, and characterize the unique example where equality holds for each sufficiently large rank.  
In fact, we solve the analogous problems for excluding arbitrary projective and affine geometries over any prime field.

For a class $\cM$ of matroids, let $h_{\cM}(n)$ denote the maximum number of elements in a simple matroid in $\cM$ of rank at most $n$. (If $\cM$ is a nonempty subclass of the $\GF(q)$-representable matroids then $h_{\cM}(n)$ is always defined, with $h_{\cM}(n) \le \tfrac{q^n-1}{q-1}$ for all $n$.) We call $h_{\cM}$ the \emph{extremal function} of $\cM$. This is often referred to as the \emph{size} or \emph{growth rate} function; our terminology here is an attempt to agree with the broader combinatorics literature. A simple matroid $M \in \cM$ with $|M| = h_{\cM}(r(M))$ is \emph{extremal} in $\cM$. A nonempty class $\cM$ of matroids is \emph{minor-closed} if it is closed under both minors and isomorphism. The following is a simplified version of the growth rate theorem of Geelen, Kung and Whittle [\ref{gkw09}].

\begin{theorem}\label{grt}
    If $\bF$ is a prime field and $\cM$ is a proper minor-closed subclass of the $\bF$-representable matroids, then either
    \begin{itemize}
        \item there is some $\alpha \in \nni$ such that $h_{\cM}(n) \le \alpha n$ for all $n$, or
        \item $\cM$ contains all graphic matroids and there is some $\alpha \in \nni$ such that $\binom{n+1}{2} \le h_{\cM}(n) \le \alpha n^2$ for all $n$. 
    \end{itemize}
\end{theorem}

We call classes of the latter type \emph{quadratic}. 
The motivation for the applications of our main result is that they may help solve the difficult problem of classifying the extremal functions of quadratic classes of representable matroids exactly.  

To state our results, we first introduce the terminology used to naturally describe the extremal functions and the extremal matroids which will occur. Fix a finite field $\bF$ and a subgroup $\Gamma$ of the multiplicative group $\bF^*$. The \emph{weight} of a vector is its number of nonzero entries. A \emph{unit vector} is a weight-$1$ vector whose nonzero entry is $1$. A \emph{$\Gamma$-frame matrix} is an $\bF$-matrix in which each column is either a weight-$0$ vector, a unit vector, or a weight-$2$ vector of the form $\gamma e_j - e_i$ for some $\gamma \in \Gamma$ and distinct unit vectors $e_i$ and $e_j$. A matroid represented by a $\Gamma$-frame matrix is a \emph{$\Gamma$-frame matroid.} The class of $\Gamma$-frame matroids is well-known to be minor-closed; see [\ref{zaslav}] for a comprehensive reference.

 Write $\cG(\Gamma)$ for the class of $\Gamma$-frame matroids, and $\cG(\Gamma)^t$ for the class of matroids having a representation $\sqbinom{P}{A}$ for some $\bF$-matrix $P$ with at most $t$ rows and some $\Gamma$-frame matrix $A$. (In this notation $\bF$ is implicit.) Note that $\cG(\{1\})^0$ is the class of graphic matroids. We will see that $\cG(\Gamma)^t$ is minor-closed, and has extremal function $f_{|\bF|,|\Gamma|,t}(n)$ defined by 
\[f_{|\bF|,|\Gamma|,t}(n) = |\bF|^t\left(|\Gamma|\tbinom{n-t}{2} + n-t\right) - \tfrac{|\bF|^t-1}{|\bF|-1}\]
for all $n \ge t$. We refer to this function $f_{q,g,t}(n)$, which is quadratic in $n$ with leading term $\tfrac{1}{2}gq^tn^2$, frequently throughout. For each $n \ge t$, there is a unique rank-$n$ extremal matroid $M$ in $\cG(\Gamma)^t$ given by $M \cong \si\left(M\sqbinom{P}{A}\right)$, where $\sqbinom{P}{A}$ includes all possible columns for which $P$ has $t$ rows and $A$ is a $\Gamma$-frame matrix with $n-t$ rows. We call this extremal matroid $\DG(n,\Gamma)^t$; it will be discussed later in more detail. 

We can now fully state the main applications of our main result (Theorem \ref{templatetech}) for excluding projective and affine geometries over all prime fields. The $p \le 3$ case differs from the general case; this is essentially because rank-$3$ projective/affine geometries can be binary/ternary frame matroids but are not frame matroids over larger fields. These results are all subject to Hypothesis \ref{structure1}.

	\begin{theorem}\label{maintwo}
		Let $t \in \nni$ and $N$ be one of $\PG(t+2,2)$ or $\AG(t+3,2)$. If $\cM$ is the class of binary matroids with no $N$-minor, then 
		\[h_{\cM}(n) = f_{2,1,t}(n) = 2^t\tbinom{n-t+1}{2} + 2^t-1\]
		for all sufficiently large $n$. Moreover, if $M$ is extremal in $\cM$ and $r(M)$ is sufficiently large, then $M \cong \DG(r(M),\{1\})^t$. 
	\end{theorem}

	\begin{theorem}\label{mainthree}
		If $t \in \nni$, the class $\cM$ of ternary matroids with no $\AG(t+2,3)$-minor satisfies
		\[h_{\cM}(n) = f_{3,2,t}(n) = 3^t(n-t)^2 + \tfrac{1}{2}(3^t-1)\] 
		for all sufficiently large $n$. Furthermore, if $M$ is extremal in $\cM$ and $r(M)$ is sufficiently large, then $M \cong \DG(r(M),\GF(3)^*)^t$. 
	\end{theorem}

    \begin{theorem}\label{mainodd}
        Let $t \in \nni$ and $N$ be either $\PG(t+1,p)$ for some prime $p \ge 3$ or $\AG(t+1,p)$ for some prime $p \ge 5$. If $\cM$ is the class of $\GF(p)$-representable matroids with no $N$-minor, then \[h_{\cM}(n) = f_{p,(p-1)/2,t}(n) = p^t\left(\tfrac{p-1}{2}\tbinom{n-t}{2} + n-t\right) + \tfrac{p^t-1}{p-1}\] 
        for all sufficiently large $n$. Moreover, if $M$ is extremal in $\cM$ and $r(M)$ is sufficiently large, then $M \cong \DG(r(M),\Gamma)^t$, where $\Gamma$ is the index-$2$ subgroup of $\GF(p)^*$.
    \end{theorem}

	Theorem~\ref{maintwo} was previously known only for $t=0$ and, in the case of projective geometries, $t = 1$ (see [\ref{gvz},\ref{heller},\ref{kmpr}]); Theorems~\ref{mainthree} and~\ref{mainodd} were unknown for all $t$. 
	They will all follow from a more general result, Theorem~\ref{bigmain}, which is also subject to Hypothesis \ref{structure1}; we state a simplified version here.
	
	\begin{theorem}\label{simplifiedmain}
		Let $\bF = \GF(p)$ be a prime field. If $\cM$ is a quadratic minor-closed class of $\bF$-representable matroids then there exists $\Gamma \le \bF^*$ and $t \in \nni$  such that $\cG(\Gamma)^t \subseteq \cM$ and $h_{\cM}(n) = f_{p,|\Gamma|,t}(n) + O(n)$. Furthermore, either 
		\begin{itemize}
			\item for all sufficiently large $n$ we have $h_{\cM}(n) = f_{p,|\Gamma|,t}(n)$ and $\DG(n,\Gamma)^t$ is the unique extremal rank-$n$ matroid in $\cM$, or 
			\item for all sufficiently large $n$, the class $\cM$ contains a simple rank-$n$ extension of $\DG(n,\Gamma)^t$. 
		\end{itemize}
	\end{theorem}

 Most of our material, when specialised to binary matroids, was originally proved in [\ref{walsh}].

\section{Preliminaries}
    
    We use the notation of Oxley [\ref{oxley}], and also write $|M|$ for $|E(M)|$ and $\elem(M)$ for $|\si(M)|$ for a matroid $M$. The rows and columns of matrices and the co-ordinates of vectors will always be indexed by sets, and thus have no inherent ordering. We write $0_A$ and $\one{A}$ for the zero and all-ones vector in $\bF^A$ respectively, and $0_{A \times B}$ for the zero matrix in $\bF^{A \times B}$, and we identify $\bF^A \times \bF^B$ with $\bF^{A \cup B}$ for disjoint $A$ and $B$. For a matrix $P = \bF^{A \times B}$, we write $P[A',B']$ for the submatrix with rows in $A'$ and columns in $B'$, and write $P[A']$ for $P[A',B]$ and $P[B']$ for $P[A,B']$ where there is no ambiguity. If $|A| = |B|$ but $A \ne B$ then the `determinant' of a matrix $P \in \bF^{A \times B}$ is only defined up to sign, and identity matrices do not make sense, but nonsingularity and $P^{-1}$ (where it exists) are well-defined. We refer to any square matrix in $\bF^{A \times B}$ whose columns are distinct unit vectors as a \emph{bijection matrix}.
    
    Let $U \subseteq \bF^E$. For a vector $u \in U$ and a set $X \subseteq E$, we write $u[X]$ for the co-ordinate projection of $u$ onto $X$, and $U[X] = \{u[X]\colon u \in U\}$. For a set $\Gamma \subseteq \bF$ (typically a multiplicative subgroup), write $\Gamma U = \{\gamma u \colon u \in U, \gamma \in \Gamma\}$.  For a matrix $P \in \bF^{E \times E}$ we denote $\{Pu\colon u \in U\}$ by $PU$. If $U$ and $W$ are additive subgroups of $\bF^E$ then we say $U$ and $W$ are \emph{skew} if $U \cap W = \{0\}$, and if they are skew subspaces with $U + W = \bF^E$ then they are \emph{complementary}; a pair of complementary subspaces gives rise to a well-defined projection map $\psi\colon \bF^E \to W$ for which $\psi(u+w) = w$ for all $u \in U$ and $w \in W$.

    \subsection*{Represented Matroids}

    Most of our arguments involve manipulation of matrices; for this purpose we will use a formalised notion of a matroid representation. Let $\bF$ be a field and $E$ be a finite set. We say two subspaces $U_1$ and $U_2$ of $\bF^E$ are \emph{projectively equivalent} if $U_1 = U_2D$ for some nonsingular diagonal matrix $D$. 
    For a field $\bF$, we define an \emph{$\bF$-represented matroid} to be a pair $(E,U)$ where $E$ is a finite set and $U$ is a subspace of $\bF^E$; two represented matroids $(E_1,U_1)$ and $(E_2,U_2)$ are equal if $E_1 = E_2$ and $U_1$ and $U_2$ are projectively equivalent, and are \emph{isomorphic} if there is a bijection $\varphi \colon E_1 \to E_2$ such that $\{(u_{\varphi(e)}\colon e \in E_1)\colon u \in U_1\}$ is projectively equivalent to $U_2$.
    
     A \emph{representation} of $M$ is an $\bF$-matrix $A$ whose row space is projectively equivalent to $U$ (that is, its rowspace is $U$ after some set of nonzero column scalings); we write $M = M(A)$. For each $X \subseteq E$, we write $r(X)$ for the dimension of the subspace $U[X]$, or equivalently $\rank(A[X])$ for any representation $A$ of $M$. Note that $r(\cdot)$ is invariant under projective equivalence so is well-defined. The pair $\tilde{M} = (E,r)$ is a matroid in the usual sense, and we call this the \emph{abstract matroid} associated with $M$; an abstract matroid $N$ is thus $\bF$-representable if and only if there is some $\bF$-represented matroid $M$ with $N = \tilde{M}$. 
    
    From here on we will be working with represented matroids exclusively, abbreviating them as just \emph{matroids}. To be precise, we define $\cG(\Gamma)^t$ in this new context to be the class of $\bF$-represented matroids of the form $M\sqbinom{P}{A}$ for some $\Gamma$-frame matrix $A$ and some matrix $P$ with at most $t$ rows. 
    
    The \emph{dual} of a represented matroid $M = (E,U)$ is defined to be $M^* = (E,U^{\perp})$, and for a set $X \subseteq E$ we define $M \del X = (E-X,U[E-X])$ and $M \con X = (M^* \del X)^*$ and define minors of $M$ accordingly; these are well-defined, and agree with the usual notions of minors and duality in the abstract matroid. An \emph{extension} of a represented matroid $M$ is a matroid $M^+$ such that $M^+ \del e = M$ for some $e \in E(M^+)$, or equivalently a matroid having a representation obtained from one of $M$ by appending a new column. Any invariant property or parameter of abstract matroids can easily be extended to represented matroids, and we define (co-)simplicity, (co-)simplification, the parameter $\elem(M)$, and the extremal function $h_{\cM}$ for a class $\cM$ of represented matroids in the obvious way. 

    \subsection*{Connectivity}
    
    Write $\lambda_M(A) = r_M(A) + r_M(E(M)-A) - r(M)$ for each $A \subseteq E(M)$. For $k \in \nni$, a matroid $M$ of rank at least $k$ is \emph{vertically $k$-connected} if for every partition $(A,B)$ of $E(M)$ with $\lambda_M(A) < k-1$, either $A$ or $B$ is spanning in $M$. (This definition is somewhat nonstandard but equivalent to the usual one.) 
    We require a theorem from [\ref{gn}], which  roughly states that the highly-connected matroids exemplify the densest members of any quadratic class. The version we state is both simplified and specialised to matroids over prime fields. 
    
    \begin{theorem}\label{connreduction}
        Let $\bF$ be a prime field and let $f(x)$ be a real quadratic polynomial with positive leading coefficient. If $\cM$ is a quadratic minor-closed class of $\bF$-represented matroids with $h_{\cM}(n) > f(n)$ for infinitely many $n \in \nni$, then for every $k \in \nni$ there is a vertically $k$-connected matroid $M \in \cM$ with $r(M) \ge k$ and $\elem(M) > f(r(M))$. 
    \end{theorem}
	
	To obtain the equality characterisation in Theorem \ref{mainsimple} as well as the bound, we need a lemma that is a variant of the above.
	
	\begin{lemma}\label{equalityhc}
		Let $\bF$ be a finite field and $f(x)$ be a real quadratic polynomial with positive leading coefficient, and let $k \in \nni$. If $\cM$ is a restriction-closed class of $\bF$-represented matroids and $h_{\cM}(n) = f(n)$ for all sufficiently large $n$, then for all sufficiently large $r$, every rank-$r$ matroid $M \in \cM$ with $\elem(M) = f(r)$ is vertically $k$-connected.
	\end{lemma}
	\begin{proof}
		Say $f(x) = ax^2 + bx + c$ where $a,b,c \in \bR$ and $a > 0$. Set $n_0 \in \nni$ so that $n_0 \ge 2k+a^{-1}$, while $h_{\cM}(n) = f(n)$ for all $n \ge n_0$, and $f$ is increasing on $[n_0,\infty)$. Let $n_1 = \max(2n_0,f(k), \tfrac{1}{2a}(|\bF|^{n_0} + a-b))$; we show that every $M \in \cM$ with $r(M) \ge n_1$ and $\elem(M) = f(r(M))$ is vertically $k$-connected.  
		
		Let $M \in \cM$ satisfy $r(M) = r \ge n_1$ and $\elem(M) = f(r)$. If $M$ is not vertically $k$-connected, then there is a partition $(A,B)$ of $E(M)$ for which $1 \le r_M(A) \le r_M(B) \le r-1$ and $r \le r_M(A) + r_M(B) < r + k$. Let $r_B = r_M(B)$ and $r_A = r_M(A)$; note that $r_B \ge \tfrac{r}{2} \ge n_0$ so we have $\elem(M|B) \le f(r_B) \le f(r-1) = a(r-1)^2 + b(r-1) + c$. 
		
		If $r_A \le n_0$ then $\elem(M|A) < |\bF|^{n_0}$ so 
		\begin{align*}
		ar^2 + br + c &= \elem(M) \\
		&\le \elem(M|A) + \elem(M|B) \\
		&< |\bF|^{n_0} + a(r-1)^2 + b(r-1) + c.
		\end{align*}
		This implies that $r < \tfrac{1}{2a}\left(|\bF|^{n_0} +a-b\right) \le n_1$, a contradiction. 
		
		If $r_A > n_0$ then $\elem(M|A) \le f(r_A)$ and we have
		\begin{align*}
		f(r) &\le \elem(M|A) + \elem(M|B) \\
		&\le a(r_A^2 + r_B^2) + b(r_A + r_B) + 2c\\
		&= a(r_A + r_B)^2 + b(r_A + r_B) + 2c - 2ar_Ar_B\\
		&< a(r+k)^2 + b(r+k) + 2c - 2an_0 (\tfrac{r}{2})\\
		&= f(r) + f(k) + ra(2k-n_0),
		\end{align*}
		where we use $r_A \ge n_0$ and $r_B \ge \tfrac{r}{2}$. This gives $ra(n_0-2k) < f(k)$ so, using $n_0 \ge 2k+a^{-1}$, we have $r < f(k) \le n_1$, again a contradiction.  
	\end{proof}

\section{Frame matroids and extensions}\label{dowlingsection}

In this section we define the extremal matroids in classes $\cG(\Gamma)^t$, and consider certain slightly larger classes.
Let $\bF$ be a finite field, $\Gamma$ be a subgroup of $\bF^*$ and $n \in \nni$. Let $B_0$ be an $n$-element set and $b_1, \dotsc, b_n$ be the unit vectors in $\bF^{B_0}$. Let \[W(n) = \{b_i \colon i \in [n]\} \cup \{-b_i + \gamma b_j\colon i,j \in [n], i < j, \gamma \in \Gamma\},\]
and $A \in \bF^{B_0 \times E}$ be a matrix whose set of columns is $W(n)$. We write $\DG(n,\Gamma)$ for any matroid isomorphic to $M(A)$ (this is a \emph{Dowling geometry} over $\Gamma$), and call any matrix obtained from such an $A$ by column scalings a \emph{standard} representation for $\DG(n,\Gamma)$. Given any $\Gamma$-frame matrix $A' \in \bF^{B \times F}$ with $\rank(A') = n$, we can remove redundant rows, rename rows, and rescale columns to obtain a matrix whose columns are all in $W(n)$; it follows that any rank-$n$ extremal matroid in $\cG(\Gamma)$ is isomorphic to $\DG(n,\Gamma)$. 

For each $t \in \nni$ and $n \ge t$, let $X$ be a $t$-element set and $B_0$ be an $(n-t)$-element set, and $A^t \in \bF^{(B_0 \cup X) \times E}$ be a matrix whose set of columns is 
\[W^t(n) =  (\bF^X \times W(n-t)) \cup (U \times \{0_{B_0}\}),\]
 where $U$ is a maximal set of pairwise non-parallel nonzero vectors in $\bF^X$. We write $\DG(n,\Gamma)^t$ for any matroid isomorphic to $M(A^t)$ for such an $A^t$, and call any matrix obtained from such an $A^t$ by column scalings a \emph{standard} representation of $\DG(n,\Gamma)^t$. It is not hard to check that, given a standard representation, rescaling rows in $B_0$ by elements of $\gamma$ yields another standard representation. Moreover, given any rank-$n$ matroid $M \in \cG(\Gamma)^t$ with $n \ge t$, we can remove/append/rename rows then rescale columns to find a representation for $M$ whose columns are all in $W^t(n)$; therefore every rank-$n$ extremal matroid in $\cG(\Gamma)^t$ is isomorphic to $\DG(n,\Gamma)^t$. We can thus determine the extremal function $h_{\cG(\Gamma)^t}(n) = |W^t(n)|$; we have $|W(n)| = |\Gamma|\binom{n}{2} + n$ and so \[h_{\cG(\Gamma)^t}(n) = |W^t(n)| = |\bF|^t |W(n-t)| + \tfrac{|\bF|^t-1}{|\bF|-1} = f_{|\bF|,|\Gamma|,t}(n),\] 
which justifies our earlier claims. 

The next lemma's proof uses the fact that $\cG(\Gamma)$ is minor-closed. 

\begin{lemma}\label{findpg}
	Let $\bF = \GF(q)$ be a finite field and $\Gamma \le \bF^*$. For all $t \in \nni$, the class $\cG(\Gamma)^t$ is minor-closed.
\end{lemma}
\begin{proof}
	 Let $M \in \cG(\Gamma)^t$, so there exist $P \in \bF^{T \times E}$ with $t$ rows and a $\Gamma$-frame matrix $Q \in \bF^{B \times E}$, such that $M = M(A)$ for $A = \sqbinom{P}{Q}$. Let $e \in E$ be a nonloop of $M$. Clearly $M(A) \del e \in \cG(\Gamma)^t$. If $Q[e] = 0$ then we can perform row-operations within $P$ and remove a row of $P$ to contract $e$ and we have $M(A) \con e \in \cG(\Gamma)^{t-1} \subseteq \cG(\Gamma)^t$. If $Q[e] \ne 0$ then $A$ is row-equivalent to a matrix $\sqbinom{P'}{Q}$ for which $P' \in \bF^{T \times E}$ satisfies $P'[e] = 0$. Then $M(A) \con e = M\sqbinom{P'}{Q'}$ for some $\Gamma$-frame matrix $Q'$ that represents $M(Q) \con e$. Therefore $M(A) \con e \in \cG(\Gamma)^t$; it follows that $\cG(\Gamma)^t$ is minor-closed. 
\end{proof}


\subsection*{Extensions}

If $x \in \bF^* - \Gamma$, then let $\DG^{(x)}(n,\Gamma)^t$ denote a matroid of the form $M(A|w)$, where $A \in \bF^{(X \cup B_0) \times E}$ is a standard representation of $\DG(n,\Gamma)^t$ and $w$ is a vector for which $w[X] = 0$ and $w[B_0]$ has weight $2$ and has nonzero entries $-1$ and $x$; this is a frame matroid over some subgroup $\Gamma'$ properly containing $\Gamma$. One can check that if $x$ and $x'$ lie in the same coset of $\Gamma$ in $\bF^*$, then $\DG^{(x)}(n,\Gamma)^t$ and $\DG^{(x')}(n,\Gamma)^t$ are isomorphic. Let $\DG^{\smallsq}(n,\Gamma)^t$ denote a matroid of the form $M(A|w)$, where $A$ is a standard representation of $\DG(n,\Gamma)^t$ and $w$ is the sum of three distinct unit vectors whose nonzero entries lie in $B_0$. From here on we write $\bF_p$ for the prime subfield of a finite field $\bF$.

\begin{lemma}\label{primesubfield}
	If $\Gamma$ is a subgroup of $\bF^*$ for some finite field $\bF$ and $\bF_p^* \not\subseteq \Gamma$, then $\DG^{\smallsq}(n+1,\Gamma)$ has a $\DG^{(x)}(n,\Gamma)$-minor for some $x \in \bF_p^* - \Gamma$. 
\end{lemma}
\begin{proof}
	Let $[A|w] \in \bF^{(X \cup B_0) \times (E \cup \{e\})}$ be a representation of $\DG^{\smallsq}(n+1,\Gamma)$ for which $A$ is a standard representation of $\DG(n+1,\Gamma)$ and $w = A[e]$ is the sum of three distinct unit vectors supported on $B_0$. 
	
	Let $r_1,r_2,r_3 \in B_0$ be the rows on which $w$ is nonzero. Let $E' \subseteq E$ be a set so that $A[r_1,E'] = 0$ and $A' = A[X \cup B_0 - r_1,E']$ is a standard representation of $\DG(n,\Gamma)$. If $-1 \notin \Gamma$, then contracting the unit column supported on $r_1$ and restricting to $E'$ yields a representation of a $\DG^{(-1)}(n,\Gamma)$-minor of $M$, as required. So we may assume that $-1 \in \Gamma$. 
	
	Since $\bF_p^* \not\subseteq \Gamma$ and $1 \in \Gamma$, there is some $\gamma \in \Gamma \cap \bF_p^*$ for which $\gamma \ne -1$ and $\gamma + 1 \notin \Gamma$. Consider a minor $M'$ of $M$ obtained by contracting a column $c$ of $A$ supported on $\{r_1,r_2\}$ for which $c[r_2] = -\gamma c[r_1]$, then restricting to $E'$. Now $M' = M(A'|w')$ where $w'$ has weight $2$ and has nonzero entries $1+\gamma$ and $1$. Thus $M' \cong \DG^{(-1-\gamma)}(n,\Gamma)$. If $-1-\gamma \notin \Gamma$ then the result holds; thus we may assume that $-1-\gamma \in \Gamma$ and so $(-1)(-1-\gamma) = 1 + \gamma \in \Gamma$, a contradiction.  
\end{proof}

\begin{lemma}\label{dowlingextension}
	Let $m \in \nni$, let $\bF$ be a finite field and $\Gamma$ be a subgroup of $\bF^*$. If $n \in \nni$ satisfies $n \ge |\bF|^2m + t + 3$ and $M$ is a simple rank-$n$ $\bF$-represented matroid that is an extension of $\DG(n,\Gamma)^t$, then either
	\begin{itemize}
		\item $M$ has a $\DG^{(x)}(m,\Gamma)^t$-minor for some $x \notin \Gamma$, or
		\item $\bF_p^* \subseteq \Gamma$ and $M$ has a $\DG^{\smallsq}(m,\Gamma)^t$-minor. 
	\end{itemize}
\end{lemma}
\begin{proof}
	Let $e$ satisfy $M \del e \cong \DG(n,\Gamma)^t$.  Let $A \in \bF^{(X \cup B_0) \times E}$ be a standard representation of $\DG(n,\Gamma)^t$ for which $M = M(A|w)$ for some $w \in \bF^{X \cup B_0}$; since $w$ is not parallel to a column of $A$, we may assume that either $w[B_0]$ has weight at least $3$, or that $w[B_0]$ has weight $2$ and its two nonzero entries $\alpha$ and $\beta$ satisfy $-\alpha\beta^{-1} \notin \Gamma$. Let $r \in B_0$ be such that $w[r] \ne 0$; by adding multiples of $r$ to the rows in $X$ we obtain a matrix $[A'|w']$ row-equivalent to $[A|w]$ for which $A'[B_0] = A[B_0]$ and $w'[X] = 0$; since $M(A) = M(A')$ we see that $A'$ is also a standard representation of $\DG(n,\Gamma)^t$. We may thus assume that $w[X] = 0$. 
	
	If $w[B_0]$ has weight $2$, then we can scale $w$ to obtain a weight-two column $\wh{w}$ whose nonzero entries are $-1$ and $x = -\alpha\beta^{-1} \notin \Gamma$. Since $m \le n$, the matrix $[A|\wh{w}]$ has a submatrix $[A'|w'] \in \bF^{(B_0' \cup X) \times E'}$ for which $w'$ is a weight-$2$ subvector of $\wh{w}$, and $A'$ is a standard representation of $\DG(m,\Gamma)^t$. We have $M[A'|w'] \cong \DG^{(x)}(m,\Gamma)$. But all unit vectors in $\bF^{B_0 \cup X}$ are columns of $A$, so $M[A'|w']$ is a minor of $M$, giving the result. 
	
	Otherwise, let $R$ be a three-element subset of $B_0$ for which $w[R]$ has no zero entries. By a majority argument, there is a $p(m-2)$-element subset $B$ of $B_0-R$ for which $w[B]$ is constant; say all its entries are $\alpha \in \bF$. Let $(B^0,B^1, \dotsc, B^{p-1})$ be a partition of $B$ into equal-sized sets. There are disjoint $(m-2)$-element subsets $C^1, \dotsc, C^p$ of $E$ for which $A[X,C^i] = 0$ and $A[B^0 \cup B^i,C^i] = \sqbinom{-J^i}{J^0}$ for each $i \in \{1, \dotsc, p-1\}$, where $J^i \in \bF^{B^i \times C^i}$ is a bijection matrix. Let $E_0 \subseteq E$ be a set for which $A[B^i,E_0] = 0$ for each $i \ge 1$, and the matrix $A_0 \in \bF^{(B^0 \cup R \cup X) \times E_0}$ is a standard representation of $\DG(m+1,\Gamma)^t$.
	By construction of the $B^i$ and $C^i$, we have $(M \con \cup_{i=1}^{p-1} C_i)|(F \cup \{e\}) = M(A_0|w_0)$, where $w_0[R] = w[R]$ and each entry of $w_0[B_0]$ is the sum of $p$ copies of $\alpha$ so is zero; that is, $w_0$ has weight $3$. Let $M_0 = M(A_0|w_0)$. 
	
	Note that $r(M_0) = m+1$. Let $\beta_0,\beta_1$ and $\beta_2$ be the nonzero entries of $w_0$. We may assume by scaling that $\beta_0 = -1$. If $\beta_1 \notin \Gamma$ then removing the row containing $\beta_2$ yields a representation of a $\DG^{(\beta_1)}(m,\Gamma)^t$-minor of $M$, so we may assume that $\beta_1 \in \Gamma$ and, symmetrically, that $\beta_2 \in \Gamma$. By scaling the rows containing $\beta_1$ and $\beta_2$ by $\beta_1^{-1}$ and $\beta_2^{-1}$ respectively, we may assume that $\beta_1 = \beta_2 = 1$. If $-1 \notin \Gamma$ then removing the row containing $\beta_0$ yields a representation of a $\DG^{(-1)}(m,\Gamma)^t$-minor of $M$. If $-1 \in \Gamma$ then scaling the row containing $\beta_0$ by $-1$ yields a representation of $\DG^{\smallsq}(m+1,\Gamma)^t$. The result now follows easily from Lemma~\ref{primesubfield}. 
\end{proof}

For each $x$, let $\xpinch{t}$ denote the closure of $\{\DG^{(x)}(n,\Gamma)^t \colon n \ge t\}$ under minors and isomorphism, and define $\sqpinch{t}$ analogously.  
For each $d \in \nni$, let $\cG(\Gamma)^t_d$ denote the closure under minors of the class of matroids of the form $M[A | D]$, where $M(A) \in \cG(\Gamma)^t$ and $D$ has $d$ columns. It is clear by the definitions that $\cG(\Gamma)^t_1$ contains $\sqpinch{t}$ and $\xpinch{t}$ for all $x$. We will later require an easy lemma characterising matroids in $\cG(\Gamma)^t_d$. 
\begin{lemma}\label{extensionprojection}
	For $t,d \in \nni$, each matroid in $\cG(\Gamma)^t_d$ is a minor of a matroid having a representation $\sqbinom{P_2}{P_0}$, where $P_2$ has $t$ rows, and $P_0$ is a matrix for which there is a matrix $P_1$ with $d$ rows such that $\sqbinom{P_1}{P_0}$ is row-equivalent to a $\Gamma$-frame matrix. 
\end{lemma}
\begin{proof}
	Let $N_0 \in \cG(\Gamma)^t_d$; we see that there is a matroid $N$ with an $N_0$-minor, $t$-element set $X$, a set $B$ and a $d$-element set $R \subseteq E(N)$ so that $N = M(A)$, where $A \in \bF^{(B \cup X) \times E(N)}$ is such that $A[B,E(N)-R]$ is a $\Gamma$-frame matrix. Let $A_N = A[B,E-R] \oplus I_R$. It is clear that $A_N[B \cup R]$ is a $\Gamma$-frame matrix, and that $A_N$ is row-equivalent to the matrix $A_N' = A_N + (A[R] \oplus 0_{R \times E})$. But $A[B] = A_N'[B]$, so $A[B]$ is obtained from a matrix row-equivalent to a $\Gamma$-frame matrix by removing $|R| = d$ rows. Since $A[X]$ has $t$ rows, the result follows with $P_2 = A[X]$ and $P_0 = A[B]$ and $P_1 = A_N'[R]$. 
\end{proof}


	 To derive our main results, we need to understand which projective and affine geometries belong to which $\cG(\Gamma)^t$, $\xpinch{t}$ and $\sqpinch{t}$. Given a subgroup $\Gamma$ of $\bF^*$ and $t \in \nni$, we write $(t,\Gamma) \preceq (t',\Gamma')$ if $(t,|\Gamma|)$ does not exceed $(t,|\Gamma'|)$ in the lexicographic order on $(\nni)^2$. This is equivalent to the statement that $|\bF|^t|\Gamma| \le |\bF|^{t'}|\Gamma'|$. 
	 
	 In the next two lemmas, we use the fact that a simple $\bF$-represented matroid is a restriction of an affine geometry if and only if it has a representation $A$ for which $\row(A)$ contains a vector with no zero entries. 
	
\begin{lemma}\label{techtwo}
	Let $\bF = \GF(2)$ and $t \in \nni$. If $N$ is one of $\AG(t+3,2)$ or $\PG(t+2,2)$, then $N \notin \cG(\{1\})^t$ but $N \in \sqpinchg{\{1\}}{t}$ and $N \in \cG(\{1\})^{t'}$ for all $t' > t$. 
\end{lemma}
\begin{proof}
	Since $\sqpinchg{\{1\}}{t}$ and $\cG(\{1\})^t$ are minor-closed and $\PG(t+2,2)$ is a minor of $\AG(t+3,2)$, it suffices to show that $\PG(t+2,2) \notin \cG(\{1\})^t$ and $\AG(t+3,2) \in \sqpinchg{\{1\}}{t} \cap \cG(\{1\})^{t+1}$. The class $\cG(\{1\})^t$ has extremal function $f_{2,1,t}(n) = 2^t\tbinom{n-t+1}{2} + 2^t-1$. Now $f_{2,1,t}(t+3) = 7 \cdot 2^t - 1 < 2^{t+3}-1$, so $\PG(t+2,2) \notin \cG(\{1\})^t$ as required. 
	
	Let $K \cong K_{2,4}$ be the complete bipartite graph with bipartition $(\{1,2,3,4\},\{5,6\})$. Let $X \subseteq \bF^{[6]}$ be the set of columns of the incidence matrix of $K$ and $w \in \bF^{[6]}$ be the characteristic vector of $\{1,2,3,4\}$. Let $T$ be a $t$-element set, and $A$ be a matrix with row-set $[6] \cup T$ whose set of columns is $\bF^T \times X$. Let $w' = (0^T,w)$ and let $M$ be obtained from $M[A|w']$ by contracting column $w'$. Since the incidence matrix of $K$ has rank $5$, we can remove a redundant row from $A|w'$ to see that $M$ is a contraction of a restriction of $\DG_{\smallsq}(t+5,\{1\})$, so $r(M) \le t+4$ and $M \in \sqpinchg{\{1\}}{t}$. 
	
	By construction, no pair of columns of $A$ add to $w'$, so $M$ is simple with $|M| = 2^t|X| = 2^{t+3}$. Moreover, one can check that $M$ has a representation $A_0$ with row set $([6]-\{1\}) \cup T$ for which $A_0[T] = A[T]$ and $A_0[5] + A_0[6]$ sum to the all-ones vector. Therefore $M$ is a simple restriction of $\AG(t+3,2)$ with $2^{t+3}$ elements, from which it follows that $M \cong \AG(t+3,2)$, so $\AG(t+3,2) \in \sqpinchg{\{1\}}{t}$. 
	
	Finally, consider a matrix $A'$ with row-set $[t+5]$ that contains as columns precisely the $v$ for which $v[[4]]$ is a column of the incidence matrix of the $4$-cycle $(1,2,3,4)$ with vertex set $[4]$. Clearly $\rank(A') = t+4$ and $M(A') \in \cG(\{1\})^{t+1} \subseteq \cG(\{1\})^{t'}$; moreover, $A'[2] + A'[4]$ is an all-ones vector, so $M(A')$ is a restriction of $\AG(t+3,2)$; since $A'$ has $2^{t+3}$ distinct columns it follows that $M(A') \cong \AG(t+3,2) \in \cG(\{1\})^{t'}$.  
\end{proof}

\begin{lemma}\label{techthree}
	Let $\bF = \GF(3)$ and $t \in \nni$. If $N \cong \AG(t+2,3)$ then $N \notin \cG(\bF^*)^t$ but $N \in \sqpinchg{\bF^*}{t}$ and $N \in \cG(\Gamma')^{t'}$ for all $(t',\Gamma') \succ (t,\Gamma)$. 
\end{lemma}
\begin{proof}
	If $A$ is a $\GF(3)$-representation of $\AG(m,3)$ for some $m \ge 2$, then removing a row of $A$ yields a representation of a matroid with an $\AG(m-1,3)$-restriction. If we had $N \in \cG(\bF^*)^t$ then we could thus remove $t$ rows from some representation of $N$ to obtain an $\bF^*$-frame matrix $A_0$ for which $M(A_0)$ has an $AG(2,3)$-restriction, and so $\AG(2,3)$ is an $\bF^*$ frame matroid. But $|AG(2,3)| = 9 = h_{\cG(\bF^*)}(3)$ so this implies that $\AG(2,3) \cong \DG(3,\bF^*)$. This is a contradiction as $\DG(3,\bF^*)$ contains a $4$-point line but $\AG(2,3)$ does not. 
	
	Define an $\bF^*$-frame matrix $Q$ with row-set $[4]$ by 
	\begin{center}
		$Q = $\begin{tabular}{c|ccccccccc|}
			\cline{2-10}
			$1$ & $0$ & $0$ & $0$ & $0$ & $0$ & $0$ & $0$ & $1$ & $1$ \\
			$2$ & $1$ & $0$ & $1$ & $1$ & $1$ & $1$ & $0$ & $0$ & $0$ \\
			$3$ & $0$ & $1$ & $0$ & $1$ & $0$ & $0$ & $1$ & $0$ & $0$ \\
			$4$ & $0$ & $0$ & $2$ & $0$ & $1$ & $2$ & $1$ & $1$ & $2$ \\
			\cline{2-10}
		\end{tabular}.
	\end{center} 
	Let $T$ be a $t$-element set and let $b_1, \dotsc, b_4$ be the unit vectors corresponding to $1,\dotsc, 4$ in $\bF^{T \cup [4]}$. Let $X$ be the set of columns of $Q$ and let $A$ be a matrix whose column set is $\bF^T \times X$. This matrix has $3^{t+2}$ columns which are nonzero and pairwise non-parallel. Let $w = b_1 + b_2 + b_3$ and $M = M(A|w)$; clearly $M$ is a restriction of $\DG_{\smallsq}(t+4,\bF^*)^t$. Note further that no two columns of $A$ span $w$, so the matroid $M_0$ obtained from $M$ by contracting column $w$ is simple with $r(M_0) \le t+3$. Furthermore, $M_0$ has a representation $\sqbinom{P}{Q'}$ for some matrix $P$ with row-set $T$ and some $Q'$ with row-set $\{2,3,4\}$ in which the sum of rows $2$ and $3$ contains no zero entries. It follows that $M_0$ is a restriction of $\AG(t+2,3)$; since $|M_0| = 3^{t+2} = |\AG(t+2,3)|$, we thus have $M_0 \cong \AG(t+2,3) \cong N$ and so $N \in \sqpinchg{\bF^*}{t}$ as required.
	
	If $(t',\Gamma') \succ (t,\Gamma)$, then $t' > t$; consider an $\bF$-matrix $A'$ with row-set $[t+3]$ containing as a column every vector $v$ for which $v[1] \ne v[2]$. Now $\si(M(A')) \cong \AG(t+3,2) = N$ and, since $A'[\{1,2\}]$ is a $\{1\}$-frame matrix up to column scalings, we have $N \in \cG(\{1\})^{t+1} \subseteq \cG(\Gamma')^{t'}$ as required. 
\end{proof}

\begin{lemma}\label{techodd}
	Let $t \in \nni$ and let $N$ be either $\PG(t+1,p)$ for some prime $p > 2$ or $\AG(t+1,p)$ for some prime $p > 3$. Let $\bF = \GF(p)$ and $\Gamma$ be the index-$2$ subgroup of $\bF^*$. Then $N \notin \cG(\Gamma)^t$ but $N \in \xpinch{t}$ for all $x \in \bF^* - \Gamma$ and $N \in \cG(\Gamma')^{t'}$ for all $(t',\Gamma') \succ (t,\Gamma)$.
\end{lemma}
\begin{proof}
	The value of the extremal function of $\cG(\Gamma)^t$ at $n = t+2$ is $f_{p,(p-1)/2,t}(t+2) = p^t\left(\tfrac{p+3}{2}\right) + \tfrac{p^t-1}{p-1}$. If $p = 3$ then this expression is $\tfrac{p^{t+2}-1}{p-1} - p^t < |N|$. If $p \ge 5$ then, using $\tfrac{p+3}{2} \le p-1$, we have $f_{p,(p-1)/2,t}(t+2) < p^{t+1} = |N|$. Since $r(N) = t+2$ in either case, we have $N \notin \cG(\Gamma)^t$. 
	
	It suffices for all $p$ to show that $\PG(t+1,2) \in \xpinch{t}$. Let $[A|w] \in \bF^{(X \cup B_0) \times (E-\{e\})}$ be a representation of $M \cong \DG^{(x)}(t+3,\Gamma)^t$, where $A$ is a standard representation of $\DG(t+3,\Gamma)$ and $w = A[e]$ is weight-$2$ vector supported on $B_0$ whose nonzero entries are $-1$ and $x$. Let $F \subseteq E$ be the set of columns of $A[B_0]$ whose support is contained in the support of $w_0$; we have $|A| = |\Gamma|+2 = \tfrac{p+3}{2}$. The lines of $M$ containing $e$ and more than one other point are the sets $L_w = \{w\}\times F$ for $w \in \bF^X$. For each $L_w$, contracting $e$ identifes the points in $L_w$; we thus lose $\tfrac{p+1}{2}$ points for each $L_w$, so 
	\begin{align*}
	\elem(M \con e) &= \elem(M)-1 - \tfrac{p+1}{2} p^t\\
	&= f_{p,(p-1)/2,t}(t+3) - \tfrac{p+1}{2} p^t \\
	&= p^t\left(\tfrac{p-1}{2}\tbinom{3}{2} + 3\right) + \tfrac{p^t-1}{p-1} - \tfrac{p+1}{2} p^t \\
	&= \tfrac{p^{t+2}-1}{p-1}.
	\end{align*}
	So $\si(M \con e)$ is a rank-$(t+2)$ matroid in $\xpinch{t}$ with $\tfrac{p^{t+2}-1}{p-1}$ elements; it follows that $\si(M \con e) \cong \PG(t+1,2)$ as required. 
	
	Let $(t',\Gamma') \succ (t,\Gamma)$. If $t' = t$ then $\Gamma = \bF^*$ and $N \in \cG(\Gamma')^{t'}$ follows from the fact that $\xpinch{t} \subseteq \cG(\bF^*)^t$. If $t' > t$, let $A'$ be an $\bF$-matrix with row-set $[t+2]$ containing as columns all vectors $v$ for which $v[1] \in \{0,1\}$; clearly $M(A') \cong \PG(t+1,p)$ has an $N$-restriction. Since $A'[1]$ is trivially a $\Gamma'$-frame matrix we thus have $N \in \cG(\Gamma')^{t+1} \subseteq \cG(\Gamma')^{t'}$. 
\end{proof}


\section{Frame Templates}

Templates were introduced in [\ref{ggwstructure}] as a means of precisely describing a class of matroids whose members are `close' to being frame matroids. We make a simplification to the original definition, where a set named `$D$' is absorbed into `$Y_0$', with no loss of generality and the definition of `conforming' is simplified accordingly; our definition is essentially identical to that given in [\ref{gvz}], [\ref{grace}], and [\ref{gvzcounterexample}]. For a field $\bF$, an \emph{$\bF$-frame template} (hereafter just a \emph{template}) is an $8$-tuple $\Phi = (\Gamma,C,X,Y_0,Y_1,A_1,\Delta,\Lambda)$, where
\begin{enumerate}[(i)]
    \item $\Gamma$ is a subgroup of $\bF^*$,
    \item $C,X,Y_0$ and $Y_1$ are disjoint finite sets,
    \item $A_1 \in \bF^{X \times (Y_0 \cup Y_1 \cup C)}$, and
    \item $\Delta$ and $\Lambda$ are additive subgroups of $\bF^{Y_0 \cup Y_1 \cup C}$ and $\bF^X$ respectively, and both are closed under scaling by elements of $\Gamma$. 
\end{enumerate}

(In the case where $\bF$ is a prime field, with which we are mostly concerned,  both $\Delta$ and $\Lambda$ are sub\emph{spaces}.) A template describes a class of matrices; we say a matrix $A' \in \bF^{B \times E}$ \emph{respects} $\Phi$ if
\begin{enumerate}[(a)]
    \item $X \subseteq B$ and $Y_0 \cup Y_1 \cup C \subseteq E$, 
    \item $A_1 = A'[X, Y_0 \cup Y_1 \cup C]$, 
    \item there is a set $Z \subseteq E-(Y_0 \cup Y_1 \cup C)$ such that $A'[X,Z] = 0$ and each column of $A'[B-X,Z]$ is a unit vector, 
    \item each row of $A'[B-X,Y_0 \cup Y_1 \cup C]$ is in $\Delta$, and
    \item the matrix $A'[B,E - (Z \cup Y_0 \cup Y_1 \cup C)]$ has the form $\sqbinom{P}{F}$, where each column of $P$ is in $\Lambda$, and $F$ is a $\Gamma$-frame matrix. 
\end{enumerate}
    Whenever we define such an $A \in \bF^{B \times E}$, we implicitly name the set $Z \subseteq E$. The structure of a matrix respecting $\Phi$ is depicted below.
\begin{center}    
\begin{tabular}{c|c|c|c|}
	\tcol{} & \tcol{} & \tcol{$Z$} & \tcol{$Y_0 \cup Y_1 \cup C$} \\
	\cline{2-4}
	$X$ & columns from $\Lambda$ & $0$ & {$A_1$}\\
	\cline{2-4}
	& $\Gamma$-frame matrix & unit columns & rows from $\Delta$\\
	\cline{2-4}	
\end{tabular}
\end{center}
A matrix $A \in \bF^{B \times E}$ \emph{conforms to $\Phi$} if there is a matrix $A' \in \bF^{B \times E}$ respecting $\Phi$ such that $A'[B,E-Z] = A[B,E-Z]$, and for each $z \in Z$ we have $A[z] = A'[z] + A'[y]$ for some $y \in Y_1$. Equivalently, a matrix conforming to $\Phi$ is one of the form $A'(I_E + H)$, where $A' \in \bF^{B \times E}$ respects $\Phi$, and $H \in \bF^{E \times E}$ is a matrix for which every nonzero entry lies in $H[Y_1,Z]$, such that every column of $H[Y_1,Z]$ is a unit vector. We call such a matrix of the form $S = I_E + H$ an \emph{$(E,Z,Y_1)$-shift matrix}; such a matrix is `lower-diagonal' and therefore nonsingular.

A matroid $M$ \emph{conforms to $\Phi$} if there is a matrix $A$ conforming to $\Phi$ for which $M = M(A) \con C \del Y_1$, or equivalently if there is a matrix $A$ respecting $\Phi$ and an $(E,Z,Y_1)$-shift matrix $S$ such that $M = M(AS) \con C \del Y_1$. Let $\cM(\Phi)$ denote the class of matroids that are isomorphic to a matroid conforming to $\Phi$, and $\cM^*(\Phi)$ denote the class of their duals. Two templates $\Phi$ and $\Phi'$ are \emph{equivalent} if $\cM(\Phi) = \cM(\Phi')$. Classes $\cM(\Phi)$ are not in general minor-closed; we write $\ol{\cM(\Phi)}$ for the closure of $\cM(\Phi)$ under minors. Note that if $\Phi$ is a template for which $|X| = t$ and $\Lambda = \bF^X$ while $C \cup Y_0 \cup Y_1 = \varnothing$, then $\cM(\Phi)$ is the class $\cG(\Gamma)^t$ described earlier. 

We can now state the hypothesis that provides the motivation for the definition of template.
This is a modified version of a hypothesis given by Geelen, Gerards, and Whittle in [\ref{ggwstructure}], and it was first stated in the following form in [\ref{gvzcounterexample}].

\begin{hypothesis}\label{structure1}
Let $\bF$ be a finite field of characteristic $p$ and $\cM$ be a minor-closed class of $\bF$-representable matroids not containing all $\GF(p)$-representable matroids. There exist integers $k$ and $m$, and finite sets $\bT$ and $\bT^*$ of $\bF$-frame templates such that 
    \begin{itemize}
    	\item $\cM$ contains $\cup_{\Phi \in \bT}\cM(\Phi)$  and $\cup_{T \in \bT^*}\cM^*(\Psi)$, and 
        \item each simple vertically $k$-connected matroid $M \in \cM$ with an $M(K_m)$-minor is in $\cup_{\Phi \in \bT}\cM(\Phi)$ or $\cup_{\Psi \in \bT^*}\cM^*(\Psi)$.
    \end{itemize}
\end{hypothesis}

While we do not assume this hypothesis in paper, our results will facilitate its future application.

\section{Taming Templates}

In this section we prove that templates can be substantially simplified with no loss of generality. Our first few lemmas give basic ways to manipulate templates without changing the class of conforming matroids. The first allows us to generically contract an appropriately structured subset of $C$. 

\begin{lemma}\label{contracttemplate}
	Let $\Phi = \templatecrap$ be a template and let $\wh{X} \subseteq X$ and $\wh{C} \subseteq C$ be sets for which $\rank(A_1[\wh{X},\wh{C}]) = |\wh{X}|$ while $A_1[X-\wh{X},\wh{C}] = 0$ and $\Delta[\wh{C}] = 0$.  Then $\Phi$ is equivalent to the template \[\Phi' = (\Gamma,C',X',Y_0,Y_1,A_1[X',C' \cup Y_0 \cup Y_1], \Lambda[X'], \Delta[C']),\]
	where $X' = X - \wh{X}$ and $C' = C-\wh{C}$. 
\end{lemma}
\begin{proof}
	For each matrix $A \in \bF^{B \times E}$ conforming to $\Phi$, let $\varphi(A)$ denote the matrix $A[B-\wh{X},E-\wh{C}]$. Note that for every $A'$ conforming to $\Phi'$ there is some $A$ conforming to $\Phi$ for which $\varphi(A) = A'$. By the hypotheses we see that $A[\wh{X},\wh{C}]$ is a rank-$|\wh{X}|$ submatrix and $A[B-\wh{X},\wh{C}] = 0$, so $M(A) \con \wh{C} = M(\varphi(A))$. It follows that 
	\[M(A) \con C \del Y_1 = M(\varphi(A)) \con (C - \wh{C}) \del Y_1 \in \cM(\Phi')\] for each $A$ conforming to $\Phi$. This gives $\cM(\Phi) \subseteq \cM(\Phi')$; the fact that $\cM(\Phi') \subseteq \cM(\Phi)$ follows from the surjectivity of $\varphi$. 
\end{proof}

The next lemma allows us to perform `row-operations' on a template.

\begin{lemma}\label{unitary}
	Let $\Phi = \templatecrap$ be a template over a field $\bF$ and let $U \in \bF^{X \times X}$ be nonsingular. Then $\Phi$ is equivalent to the template $\Phi' = (\Gamma,C,X,Y_0,Y_1,UA_1,\Delta,U\Lambda)$.
\end{lemma}
\begin{proof}
	By linearity $U\Lambda$ is an additive subgroup of $\bF^X$ that is closed under $\Gamma$-scalings. Let $A \in \bF^{B \times E}$ respect $\Phi$ and $S$ be an $(E,Z,Y_1)$-shift matrix. Let $\wh{U} = U \oplus I_{B-X}$. Then $\wh{U}A$ respects $\Phi'$  and $AS$ is row-equivalent to $\wh{U}AS$. Thus for each matrix conforming to $\Phi$ there is a row-equivalent matrix conforming to $\Phi'$, so $\cM(\Phi) = \cM(\Phi')$.  
\end{proof}

The third let us project $\Delta$ using certain rows of $A_1$. 

\begin{lemma}\label{makeskew}
	Let $\Phi = \templatecrap$ be a template over a field $\bF$ and let $(X_0,X_1)$ be a partition of $X$ for which $\Lambda[X_1] = 0$. Let $W = \rowspace(A_1[X_1])$ and $V$ be a complementary subspace of $W$ in $\bF^{C \cup Y_0 \cup Y_1}$. Let $\psi \colon \bF^{C \cup Y_0 \cup Y_1} \to V$ be the projection map $w+v \mapsto v$.  Then $\Phi$ is equivalent to the template $\Phi' = (\Gamma,C,X,Y_0,Y_1,A_1,\Lambda,\psi(\Delta))$. 
\end{lemma} 
\begin{proof}
	By linearity, the set $\psi(\Delta)$ is an additive subgroup closed under $\Gamma$-scalings. If $A \in \bF^{B \times E}$ respects $\Phi$ then let $A'$ be the matrix obtained from $A$ by applying $\psi$ to each row $u \in \Delta$ of $A[B-X,C \cup Y_0 \cup Y_1]$. Clearly $A'$ respects $\Phi$, and $A'$ is row-equivalent to $A$, since we can obtain $A'$ from $A$ by adding elements of $\row(A[X_1])$ to rows in $B-X$. Therefore every matrix respecting $\Phi$ is row-equivalent to one respecting $\Phi'$, so $\Phi$ and $\Phi'$ are equivalent. 
\end{proof}

The next lemma generically simplifies the structure of $\Delta$. 
\begin{lemma}\label{magic}
    Every template over a finite field $\bF$ is equivalent to a template $\Phi = \templatecrap$ for which there exists $C' \subseteq C$ such that $\Delta = \Gamma (\bF_p^{C'}) \times \{0\}^{(C-C') \cup Y_0 \cup Y_1}$.
\end{lemma}
\begin{proof}
    Let $\Phi' = (\Gamma,C',X',Y_0,Y_1,A_1',\Delta',\Lambda')$ be a template over a finite field $\bF$. Let $D$ be a generating set for $\Delta'$, and let $A_{\Delta} \in \bF^{\wh{X} \times (Y_0 \cup Y_1 \cup C)}$ be a matrix whose set of rows is $D$, where $\wh{X}$ is a $|D|$-element set. Let $\wh{C}$ be a set of size $|\wh{X}|$ and let $P \in (\bF_p)^{\wh{X} \times \wh{C}}$ be nonsingular. Let $C = C' \cup \wh{C}$ and $X = X' \cup \wh{X}$. Let $\Lambda = \Lambda' \times \{0\}^{\wh{X}}$ and $\Delta = \Gamma(\bF_p^{\wh{C}}) \times \{0\}^{Y_0 \cup Y_1 \cup C'}$. Finally let 
    \begin{center}$A_1 = $ \begin{tabular}{c|c|c|}
    \tcol{} & \tcol{$Y_0 \cup Y_1 \cup C'$} & \tcol{$\wh{C}$} \\
    \cline{2-3}
    $\wh{X}$ & $A_{\Delta}$ & $P$ \\
    \cline{2-3}
    $X$ & $A_1'$ & $0$\\
    \cline{2-3}
    \end{tabular}.\end{center}
    Note that $\Delta'$ and $\Lambda'$ are additive subgroups closed under $\Gamma$-scalings; we argue that the template $\Phi = \templatecrap$, which satisfies the required condition by choice of $\Delta$, is equivalent to $\Phi'$. 
    
    Define a map $\psi\colon \Gamma (\bF_p^{\wh{C}}) \to \Delta'$ by $\psi(w) = wP^{-1}A_{\Delta}$. Note that $\psi(w)$ is some $\Gamma$-scaling of an $\bF_p$-linear combination of vectors in $D \subseteq \Delta'$ so has range contained in $\Delta'$; moreover, since $D$ is a generating set, for every $u \in \Delta'$ there is some $w \in \bF_p^{\wh{C}}$ for which $\psi(w) = u$; thus $\psi$ is surjective. 
    
    Let $\wh{\Phi} = (\Gamma,C,X,Y_0,Y_1,A_1,\Delta' \times \{0\}^{\wh{C}},\Lambda)$. By Lemma~\ref{contracttemplate}, the templates $\wh{\Phi}$ and $\Phi'$ are equivalent. Let $A \in \bF^{B \times E}$ respect $\Phi$. Each row of the submatrix $A[B-X,C \cup Y_0 \cup Y_1]$ has the form $(0_{C' \cup Y_0 \cup Y_1},w)$ where $w \in \Gamma(\bF_p^{\wh{C}})$; let $\varphi(A)$ be the matrix obtained by replacing each such row with the row $(\psi(w),0^{\wh{C}}) \in \Delta$. It is clear that $\varphi(A)$ respects $\wh{\Phi}$; moreover, by the surjectivity of $\psi$, for every $\wh{A}$ respecting $\wh{\Phi}$ there is a matrix $A$ respecting $\Phi$ for which $A' = \psi(A)$. Finally, the matrices $A$ and $A'$ are row-equivalent by construction of $\psi$, so for any $(E,Z,Y_1)$-shift matrix $S$ the matrices $AS$ and $\psi(A)S$ are row-equivalent. Therefore $\cM(\Phi) = \cM(\wh{\Phi})$. Since $\wh{\Phi}$ is equivalent to $\Phi'$, the lemma follows. 
     \end{proof}

 We say a template $\Phi = \templatecrap$ over $\bF$ is \emph{$Y$-reduced} if $\Delta[C] = \Gamma(\bF_p^C)$ and $\Delta[Y_0 \cup Y_1] = \{0\}$, and there is a partition $(X_0,X_1)$ of $X$ for which $\bF_p^{X_0} \subseteq \Lambda[X_0]$ and $\Lambda[X_1] = \{0\}$.
 
\begin{lemma}
	Every template is equivalent to a $Y$-reduced template. 
\end{lemma}
\begin{proof}
	Let $\Phi = \templatecrap$ be a template over a finite field $\bF$ with prime subfield $\bF_p$. We may assume by Lemma~\ref{magic} that there is a partition $(C_0,C_1)$ of $C$ for which $\Delta = \Gamma(\bF_p^{C_0}) \times \{0_{C_1 \cup Y_0 \cup Y_1}\}$. Now $A_1$ is row-equivalent to a matrix 
	\begin{center}$A_1' = $ \begin{tabular}{c|c|c|c|}
	\tcol{} & \tcol{$C_0$} & \tcol{$C_1$} & \tcol{$Y_0 \cup Y_1$} \\
	\cline{2-4}
	$X-\wh{X}$ & \multirow{2}{*}{$P_1$} & $Q$ & \multirow{2}{*}{$P_2$} \\
	\cline{3-3}
	$\wh{X}$ & & $0$ & \\
	\cline{2-4}
	\end{tabular}\end{center}
	where $\wh{X} \subseteq X$ and $\rank(Q) = |X-\wh{X}|$. Let $U \in \bF^{X \times X}$ be nonsingular with $A_1' = UA_1$. By Lemma~\ref{unitary}, $\Phi$ is equivalent to the template \[\Phi' = (\Gamma,C,X,Y_0,Y_1,A_1',U\Lambda,\Gamma (\bF_p^{C_0}) \times \{0_{C_1 \cup Y_0 \cup Y_1}\}).\] 
	
	Let $\wh{A}_1 = A_1'[\wh{X},C_0 \cup Y_0 \cup Y_1]$. Let  $\wh{\Lambda} = (U_{\Lambda})[\wh{X}]$. Let $\wh{\Delta} = \Gamma(\bF_p^{C_0}) \times \{0\}^{Y_0 \cup Y_1}$. Let $\wh{\Phi} = (\Gamma,C_0,\wh{X},Y_0,Y_1,\wh{A}_1,\wh{\Lambda},\wh{\Delta})$. By Lemma~\ref{contracttemplate} we have $\cM(\Phi') = \cM(\wh{\Phi})$. Finally, by mapping a maximal linearly independent subset of $\wh{\Lambda}$ to a set of unit vectors, we see that there is a nonsingular matrix $\wh{U}_{\Lambda} \in \bF^{\wh{X} \times \wh{X}}$ and a partition $(\wh{X}_0,\wh{X}_1)$ of $\wh{X}$ for which the additive subgroup $\Lambda' = \wh{U}_\Lambda \wh{\Lambda}$ satisfies $\Lambda'[\wh{X}_1] = \{0\}$ and contains all unit vectors supported in $\wh{X}_0$, which implies that $\bF_p^{\wh{X}_0} \subseteq  \Lambda'[\wh{X}_0]$. By Lemma~\ref{unitary}, we see that $\wh{\Phi}$, and therefore $\Phi$, is equivalent to the $Y$-reduced template $(\Gamma,C_0,\wh{X},Y_0,Y_1,\wh{U}_{\Lambda}\wh{\Lambda},\wh{\Delta})$. 
\end{proof}

 A template $\Phi = \templatecrap$ over a field $\bF$ is \emph{reduced} if there is a partition $(X_0,X_1)$ of $X$ such that 
\begin{itemize}
	\item $\Delta = \Gamma(\bF_p^C \times \Delta')$ for some additive subgroup $\Delta'$ of $\bF^{Y_0 \cup Y_1}$, 
	\item $\bF_p^{X_0} \subseteq \Lambda[X_0]$ while $\Lambda[X_1] = \{0\}$ and $A_1[X_1,C] = 0$, and
	\item the rows of $A_1[X_1]$ are a basis for a subspace skew to $\Delta$.
\end{itemize}

\begin{lemma} \label{equivalent}
	Every template is equivalent to a reduced template. 
\end{lemma}
\begin{proof}
	Let $\Phi = \templatecrap$ be a template over a field $\bF$. We may assume that $\Phi$ is $Y$-reduced; let $(X_0,X_1)$ be the partition of $X$ for which $\Lambda[X_1] = \{0\}$ and $\bF_p^{X_0} \subseteq \Lambda[X_0]$. Let $Y = Y_0 \cup Y_1$. By applying elementary row-operations to $A_1$ without adding any multiples of rows in $X_0$ to rows in $X_1$, we obtain a matrix 
	\begin{center}$A_1' = $ \begin{tabular}{c|c|c|c|}
	\tcol{} & \tcol{$C'$} & \tcol{$C''$} & \tcol{$Y$} \\
	\cline{2-4}
	$X_1''$ & \multicolumn{2}{|c|}{$0$} & \multirow{3}{*}{$A_Y$} \\
	\cline{2-3}
	$X_1'$ & $Q$ & \multirow{2}{*}{$P_1$} & \\
	\cline{2-2}
	$X_0$ & $0$ & & \\
	\cline{2-4}
	\end{tabular},\end{center}
	where $(C',C'')$ and $(X_1',X_1'')$ are partitions of $C$ and $X$ respectively, and $Q$ is a nonsingular matrix. Let $U \in \bF^{X \times X}$ be a nonsingular matrix for which $U[X_1,X_0] = 0$ and $A_1' = UA_1$. By Lemma~\ref{unitary}, the template $\Phi$ is equivalent to $\Phi' = (\Gamma,C,X,Y_0,Y_1,A_1',U\Lambda,\Delta)$.   
	
	 Define a linear map $\psi \colon \Delta \to \bF^{C \cup Y}$ by \[\psi(w) = (0_{C'},w[C''],w[C']Q^{-1} A_Y[X_1'])\] and let $\Delta'' = \psi(\Delta)$. Note that $\Delta'' = \Gamma(\bF_p^{C''} \times \Delta_0)$ for some additive subgroup $\Delta_0$ of $\bF^Y$. Let $\Phi'' = (\Gamma,C,X,Y_0,Y_1,A_1',U_\Lambda,\Delta')$.  For each matrix $A' \in \bF^{B \times E}$ respecting $\Phi'$, let $A' \in \bF^{B \times E}$ be obtained by replacing each row $w \in \Delta$ of $A'[B-X,C \cup Y_0 \cup Y_1]$ with the row $\psi(w)$. Now $A''$ both respects $\Phi''$ and is row-equivalent to $A'$; thus each matrix respecting $\Phi'$ is row-equivalent to a matrix respecting $\Phi''$, so $\Phi'$ and $\Phi''$ are equivalent. Let $\wh{X} = X_0 \cup X_1''$ and $\wh{C} = C''$. Let $\wh{\Delta} = \Delta[\wh{C} \cup Y]$ and $\wh{A}_1 = A_1'[\wh{X},\wh{C} \cup Y]$. Since $\Delta''[C'] = 0$,  Lemma~\ref{contracttemplate} implies that the template $\Phi''$ is equivalent to $\wh{\Phi} = (\Gamma,\wh{C},\wh{X},Y_0,Y_1,\wh{A}_1,\wh{\Delta},(U\Lambda)[\wh{X}]).$ 
	 
	 Since $U[X_1,X_0] = 0$ and $\Lambda[X_1] = \{0\}$ while $\bF_p^{X_0} \times \{0_{X_1}\} \subseteq \Lambda$, we know that $\Lambda$ contains a basis for $\bF^{X_0} \times \{0_{X_1}\}$ and so $U\Lambda$ does also, and moreover $U{\Lambda}[\wh{X}]$ contains a basis for $\bF^{X_0} \times \{0_{X_1''}\}$. Let $U' \in \bF^{\wh{X} \times \wh{X}}$ be a nonsingular matrix mapping this basis to the standard basis; therefore the set $\wh{\Lambda} = U'(U\Lambda[\wh{X}])$ satisfies $\wh{\Lambda}[X_1''] = \{0\}$ and $\bF_p^{X_0} \subseteq \wh{\Lambda}[X_0]$. By Lemma~\ref{unitary} the template $\wh{\Phi}$ is equivalent to $\wh{\Phi}' = (\Gamma,\wh{C},\wh{X},Y_0,Y_1,U'\wh{A}_1,\wh{\Lambda},\wh{\Delta})$. 
	 
	 Let $V$ be a complementary subspace of $W = \row(A_1[X_1''])$ in $\bF^{\wh{C} \cup Y_0 \cup Y_1}$ and let $\varphi\colon \bF^{\wh{C} \cup Y_0 \cup Y_1} \to V$ be the associated projection map $\varphi(v+w) \mapsto v$. Since $\wh{A}_1[X_1'',\wh{C}] = 0$ we have $W[\wh{C}] = 0$ and therefore for each $u \in \bF^{\wh{C}}$ and $v \in \bF^{\wh{Y}}$ we have $\varphi(u,v) = (u,\varphi(v))$ and so $\varphi(\wh{\Delta}) = \Gamma(\bF_p^{\wh{C}} \times \wh{\Delta}_0)$ for some additive subgroup $\Delta_0 \subseteq \bF^Y$. By construction $\varphi(\wh{\Delta})$ is skew to $\row(\wh{A}_1)[X_1'']$, and moreover by Lemma~\ref{makeskew} the template $\wh{\Phi}$ is equivalent to the template $\wh{\Phi}' = (\Gamma,\wh{C},\wh{X},Y_0,Y_1,\wh{A}_1,\varphi(\wh{\Delta}),\wh{\Lambda})$.

	 The construction of $\wh{\Phi}'$ is such that $\varphi(\wh{\Delta})$ and $\wh{\Lambda}$ have the required structure for a reduced template with respect to the partition $(X_0,X_1'')$ of $\wh{X}$, save the property that the rows of $\wh{A}_1[X_1'']$ are themselves linearly independent. Since $\Lambda[X_1''] = \{0\}$, this property can easily be obtained by considering a maximal set $\wh{X}_1 \subseteq X_1''$ for which $\wh{A}_1[\wh{X}_1]$ has linearly independent rows, and then restricting $\wh{\Lambda}$ and $\wh{A}_1$ to just the rows in $X_0 \cup \wh{X}_1$. Doing so yields a template $\Psi$ with the property that every matrix respecting $\Psi$ is row-equivalent to one respecting $\wh{\Phi}'$; it follows that the reduced template $\Psi$ is equivalent to $\wh{\Phi}'$ and therefore to $\Phi$. 
\end{proof}

	\begin{corollary}\label{strongstructure}
	The sets $\bT$ and $\bT^*$ of templates given by Hypothesis~\ref{structure1} can be taken to be contain only reduced templates. 
	\end{corollary}


\section{Density and Subclasses}

Templates are especially nice over prime fields due to the fact that $\Lambda$ and $\Delta$ are subspaces; in this section we investigate them further. For a template $\Phi = \templatecrap$, define the \emph{complexity} of $\Phi$ by $c(\Phi) = |X \cup Y_0 \cup Y_1 \cup C|$. We use this measure in the next two lemmas to bound the density of matroids conforming and co-conforming to reduced templates. The primal bound is quadratic in rank.

\begin{lemma}\label{templatedensity}
	Let $\Phi = \templatecrap$ be a reduced template over a prime field $\bF = \GF(p)$ with $\dim(\Lambda) = t$ and $c(\Phi) = c$. Then every matroid $M \in \cM(\Phi)$ satisfies \[\elem(M) \le f_{p,|\Gamma|,t}(r(M)) + p^{t+1}c(r(M)\\ + c).\] 
\end{lemma}
\begin{proof}
	Let $M = M(AS)\con C \del Y_1 \in \cM(\Phi)$, where $A \in \bF^{B \times E}$ respects $\Phi$ and $S$ is an $(E,Z,Y_1)$-shift matrix. Note that $r(M) \ge \rank(AS) - |C \cup Y_1| \ge \rank(A) - c$, and since $\rank(A)$ is at least the number of distinct columns in the submatrix $A[B-X,Z]$, we see that $A[B-X,Z]$ has at most $r(M) + c$ distinct columns. Every column of $(AS)[Z]$ is the sum of a column of $A[Y_1]$ and a column of $A[Z]$, so $(AS)[Z]$ has at most $|Y_1|(r(M)+c)$ distinct columns and $(AS)[Z \cup Y_0]$ has at most $|Y_1|(r(M)+c) + |Y_0| \le c(r(M)+c)$ distinct columns. Thus $\elem(M|(Y_0 \cup Z)) \le \elem(M(AS)|(Y_0 \cup Z)) \le  c(r(M)+c)$. Let $F = E - (C \cup Y_0 \cup Y_1 \cup Z)$. We have $M(AS)|F \in \cG(\Gamma)^t$, so $\elem(M|F) \le \elem(M(AS)|F) \le f_{p,|\Gamma|,t}(\rank(AS)) \le f_{p,|\Gamma|,t}(r(M)+c)$. Now for all $x$ we have 
	\begin{align*}
	f_{p,|\Gamma|,t}(x+c) - f_{p,|\Gamma|,t}(x) &= p^t|\Gamma|((x-t)c + \tfrac{1}{2}(c^2+c)) \\
	&\le p^t(p-1)c(x+c),
	\end{align*}
	since $|\Gamma| \le p-1$ and $c \le c^2$. Combining the above estimates we have 
	\begin{align*}
	\elem(M) &\le \elem(M|F) + \elem(M|(Y_0 \cup Z)) \\
	&\le f_{p,|\Gamma|,t}(r(M)+c) + c(r(M)+c) \\
	&\le f_{p,|\Gamma|,t}(r(M)) + p^t(p-1)c(r(M) + c) + c(r(M)+c)\\
	&\le f_{p,|\Gamma|,t}(r(M)) + p^{t+1}c(r(M)+c),
	\end{align*}
	giving the bound. 
\end{proof}

For the dual bound, which is linear in rank, we need an easy lemma bounding the density of the dual of a frame matroid. (The lemma applies to any $\Gamma$-frame matroid over any field.)

\begin{lemma}\label{coframedensity}
	If $M^*$ is a frame matroid, then $\elem(M) \le 3r(M)$.
\end{lemma}
\begin{proof}
	We may assume that $M$ is simple. Let $A$ be a frame representation of $M^*$ with $r^*(M)$ rows. If some row of $A$ has weight less than $3$ then $M^*$ has a coloop or series pair so $M$ is not simple. Thus $A$ has at least $3r^*(M)$ nonzero entries, so the number $|M|$ of columns of $A$ is at least $\tfrac{3}{2}r^*(M)$. Therefore $\tfrac{r(M)}{\elem(M)} = \tfrac{|M|-r^*(M)}{|M|} = 1 - \tfrac{r*(M)}{|M|} \ge \tfrac{1}{3}$. 
\end{proof}

\begin{lemma}\label{dualdensity}
	Let $\Phi$ be a template over a finite field $\bF$. If $M \in \cM^*(\Phi)$ then $\elem(M) \le |\bF|^{c(\Phi)}(3r(M) + 6c(\Phi)+1)$. 
\end{lemma}
\begin{proof}
	Let $\Phi = \templatecrap$ and $c = c(\Phi)$. Let $M^* \in \cM(\Phi)$, so there exists $A \in \bF^{B \times E}$ respecting $\Phi$ and an $(E,Z,Y_1)$-shift matrix $S$ for which $M = (M(AS) \con C \del Y_1)^*$. Let $A'$ be obtained from $A$ by replacing all rows in $X$ and columns in $C \cup Y_0 \cup Y_1$ by zero; clearly $A' = A'S$ is a $\Gamma$-frame matrix and $\rank((A-A')S) \le \rank(A-A') \le c$. 
	
	Therefore $A' = AS + P$ for a matrix $P$ of rank at most $c$. It follows that $M(A')$ and $M(AS)$ have $\bF$-representations that agree on all but at most $c$ rows, and thus that there is a matroid $\wh{M}$ and a pair of disjoint $c$-element sets $T_1,T_2 \subseteq E(\wh{M})$ such that $\wh{M} \con T_1 \del T_2 = M(A')$ and $\wh{M} \del T_1 \con T_2 = M(AS)$. Let $N = \wh{M} \con C \del Y_1$, so $N \del T_1 \con T_2 = M^*$ and $N \con T_1 \del T_2$ is minor of $M(A')$, so is a $\Gamma$-frame matroid. By Lemma~\ref{coframedensity} we have $\elem(N^* \del T_1 \con T_2) \le 3r(N^* \del T_1 \con T_2)$ and so, since $\elem(M_0) \le |\bF|^{|H|}(\elem(M_0 \con H) +1)$ for every set $H$ in an $\bF$-represented matroid $M_0$, we have 
	\[\elem(N^* \con T_1 \del T_2) \le \elem(N^* \del T_1) \le |\bF|^c(3r(N^* \del T_1 \con T_2)+1) \le |\bF|^c(3r(N^*) + 1).\] 
	But $M = N^* \del T_2 \con T_1$ and $r(N^*) \le r(M) + |T_1 \cup T_2| = r(M) + 2c$, so this gives $\elem(M) \le |\bF|^c(3(r(M)+2c) + 1)$, as required.  
\end{proof}
The next lemma essentially states that, for a reduced template $\Phi$, the class $\ol{\cM(\Phi)}$ contains all matroids whose representation is obtained from a $\Gamma$-frame matrix by appending $\dim(\Lambda)$ rows and $\dim(\Delta)$ columns.

\begin{lemma}\label{subclass}
	If $\Phi = \templatecrap$ is a reduced template over a prime field and $(t,d) = (\dim(\Lambda),\dim(\Delta))$, then $\cG(\Gamma)_d^t \subseteq \ol{\cM(\Phi)}$. 
\end{lemma}
\begin{proof}
	Let $(X_0,X_1)$ be the partition of $X$ certifying that $\Phi$ is reduced; note that $|X_0| = t$. Let $N_0 \in \cG(\Gamma)^t_d$; by Lemma~\ref{extensionprojection} there is a matroid $N$ with an $N_0$-minor, a set $B_0$, a $d$-element set $R$ and matrices $P_1 \in \bF^{R \times F}$ and $P_2 \in \bF^{X_0 \times F}$ and $P_0 \in \bF^{B_0 \times F}$ such that $\sqbinom{P_1}{P_0}$ is row-equivalent to a $\Gamma$-frame matrix, while $N = M\sqbinom{P_2}{P_0}$. We show that $N \in \ol{\cM(\Phi)}$. 
	
	Let $Y = Y_0 \cup Y_1$. Let $W \in \bF^{R \times (C \cup Y)}$ be a matrix with rowspace $\Delta$; note, since $|R| = d$, that $W$ has linearly independent rows. Since $A_1[X_1]$ has row space skew to $\Delta$ and has linearly independent rows, we see that $\sqbinom{A_1[X_1]}{W}$ also has linearly independent rows. Let $\wh{C} \subseteq C \cup Y$ be such that the matrix $Q  = \sqbinom{A_1[X_1]}{W}[\wh{C}]$ is nonsingular. So $|\wh{C}| = d + |X_1|$, and since $\Delta = \bF^C \times \Delta[Y]$, we must have $C \subseteq \wh{C}$. Since $Q$ is nonsingular, there is a matrix $P_2' \in \bF^{X_0 \times F}$ for which the matrices
	\begin{center}$Q_1 = $
		\begin{tabular}{c|c|c|}
			\tcol{} & \tcol{$F$} & \tcol{$\wh{C}$} \\
			\cline{2-3}
			$X_0$ & $P_2'$ & $A_1[X_0,\wh{C}]$ \\
			\cline{2-3}
			$X_1$ & $0$ & \multirow{2}{*}{$Q$} \\
			\cline{2-2}
			$R$ & $P_1$ & \\
			\cline{2-3}
		\end{tabular} \text{ and } $Q_2 = $
		\begin{tabular}{c|c|c|}
			\tcol{} & \tcol{$F$} & \tcol{$\wh{C}$} \\
			\cline{2-3}
			$X_0$ & $P_2$ & $0$ \\
			\cline{2-3}
			$X_1$ & $0$ & \multirow{2}{*}{$Q$} \\
			\cline{2-2}
			$R$ & $P_1$ & \\
			\cline{2-3}
		\end{tabular}
	\end{center}
	are row-equivalent. Let $C_i = \wh{C} \cap Y_i$ for $i \in \{0,1\}$, so $\wh{C} = C \cup C_0 \cup C_1$. We essentially wish to contract $C_1$ from a matroid conforming to $\Phi$, but since the columns in $C_1$ must be deleted, we must `copy' its entries using $Z$. Let $Z$ be a copy of the set $C_1$, let $\{c,d\}$ be a $2$-element set, and consider the matrix 
	\begin{center}
		$A = $ \begin{tabular}{c|c|c|c|c|}
			\tcol{} &  \tcol{$F$} & \tcol{$c$} &\tcol{$Z$} & \tcol{$C \cup Y$}\\
			\cline{2-5}
			$X_0$  & $P_2'$ & \multirow{4}{*}{$0$} & \multirow{4}{*}{$0$} & $A_1[X_0]$ \\
			\cline{2-2} \cline{5-5}
			$X_1$ & $0$ & & & $A_1[X_1]$ \\
			\cline{2-2} \cline{5-5}
			$R$ & $P_1$ & & & $W$ \\
			\cline{2-2} \cline{5-5}
			$B_0$ & $P_0$ & & &\multirow{2}{*}{$0$} \\
			\cline{2-4}
			$d$  & $0$ & $1$ & $\one{Z}$ & \\
			\cline{2-5}
		\end{tabular},
	\end{center}
	where $\one{Z}$ is the all-ones vector in $\bF^Z$. Since $\sqbinom{P_1}{P_0}$ is row-equivalent to a $\Gamma$-frame matrix and $\rowspace(W) \subseteq \Delta$, we see that $A$ is row-equivalent to a matrix $A'$ respecting $\Phi$. Let $E$ be the set of column indices of $A$. Recall that $Z$ is a copy of $C_1 \subseteq Y_1$; let $S$ be the $(E,Z,Y_1)$-shift matrix so that $AS$ is obtained from $A$ by adding each column of $A[C_1]$ to its corresponding column in $A[Z]$. Thus
	\begin{center}
		$AS = $ \begin{tabular}{c|c|c|c|c|}
			\tcol{} &  \tcol{$F$} & \tcol{$c$} &\tcol{$Z$} & \tcol{$C \cup Y$}\\
			\cline{2-5}
			$X_0$  & $P_2'$ & \multirow{4}{*}{$0$} & \multirow{3}{*}{$V$} & $A_1[X_0]$ \\
			\cline{2-2} \cline{5-5}
			$X_1$ & $0$ & & & $A_1[X_1]$ \\
			\cline{2-2} \cline{5-5}
			$R$ & $P_1$ & & & $W$ \\
			\cline{2-2} \cline{4-5}
			$B_0$ & $P_0$ & & $0$ &\multirow{2}{*}{$0$} \\
			\cline{2-4}
			$d$  & $0$ & $1$ & $\one{Z}$ & \\
			\cline{2-5}
		\end{tabular},
	\end{center}
	where $V$ is a copy of $\sqbinom{A_1}{W}[C_1]$. So  $(AS)[R \cup X_0 \cup X_1,F \cup Z \cup C \cup C_0]$ is a copy of the matrix $Q_1$ defined earlier. Let
	\[M_0  = (M(AS) \con (\{c\} \cup Z \cup C \cup C_0))|F.\]
	Let $\wh{A} = AS[F \cup \{c\} \cup Z \cup C \cup C_0]$. By construction of $P_2'$ we can perform row-operations on $\wh{A}$ to replace $P_2'$ by $P_2$ and replace the submatrix $\wh{A}[X_0,Z \cup C \cup C_0]$ with zero. Then one can contract the set $\{c\} \cup Z \cup C \cup C_0$ from $M(\wh{A})$ by first removing from $\wh{A}$ row $d$ and column $c$, then removing all columns in $Z \cup C \cup C_0$ and all rows in $X_1 \cup R$. Thus $M_0 = M\sqbinom{P_2}{P_0} = N$. Since $N_0$ is a minor of $N = M_0$ and $M_0$ is a minor of $M(AS) \con C \del Y_1$, we have $N_0 \in \ol{\cM(\Phi)}$ and the result follows.
\end{proof}

\section{The Main Result}

We now prove our main result; a reduced template in which $\dim(\Lambda) = t$ either describes a `degenerate' class, a subclass of $\cG(\Gamma)^t$, or a class whose minor-closure contains $\sqpinch{t}$ or some $\xpinch{t}$ .

\begin{theorem}\label{templatetech}
	Let $\Phi = \templatecrap$ be a reduced template over a prime field $\bF$ and let $t = \dim(\Lambda)$ and $c = c(\Phi)$. Either 
	\begin{enumerate}
		\item\label{m0} $\cM(\Phi)$ contains no vertically $(c+1)$-connected matroid of rank at least $c+1$, 
		\item\label{m1} $\ol{\cM(\Phi)} \subseteq \cG(\Gamma)^t$,
		\item\label{m2} $\Gamma = \bF^*$ and $\sqpinch{t} \subseteq \ol{\cM(\Phi)}$, or
		\item\label{m3} $\Gamma \ne \bF^*$ and $\xpinch{t} \subseteq \ol{\cM(\Phi)}$ for some $x \in \bF^* - \Gamma$. 
	\end{enumerate}
\end{theorem}
\begin{proof}	
	Suppose that (\ref{m2}) and (\ref{m3}) do not hold. By Lemma~\ref{dowlingextension}, there is thus a matroid $N \in \cG(\Gamma)^t$ (in fact, of the form $\DG(n,\Gamma)^t$) with $r(N) \ge 2|X|$, such that no simple rank-$r(N)$ extension of $N$ is in $\ol{\cM(\Phi)}$.  Since every such extension is in $\cG(\Gamma)^t_1$, we may assume that $\cG(\Gamma)^t_1 \not\subseteq \ol{\cM(\Phi)}$ and thus, by Lemma~\ref{subclass}, that $\Delta = \{0\}$. Since $\dim(\Delta) \ge |C|$ in a reduced template, we also have $C = \varnothing$. 
	
	Let $(X_0,X_1)$ be the partition of $X$ certifying that $\Phi$ is reduced, and let $h = |X_1|$. Since $N \in \cG(\Gamma)^t$ we have $N = M\sqbinom{P_0}{Q}$ for some $P_0 \in \bF^{X_0 \times F}$ and some $\Gamma$-frame matrix $Q \in \bF^{B_0 \times F}$, where $F = E(N)$ and $B_0$ satisfies $|B_0| > |X-X_0| = |X_1|$. We may assume that $\rank\sqbinom{P_0}{Q} = |B_0| + |X_0|$, as otherwise we can remove redundant rows and rescale columns. Suppose now that (\ref{m0}) does not hold. 
	
	\begin{claim}
		$\col(A_1[X_1,Y_0]) \subseteq \col(A_1[X_1,Y_1])$. 
	\end{claim}
	\begin{proof}[Proof of claim:]
		Suppose not. Since $\Lambda[X_1] = 0$, every matrix $A \in \bF^{B \times E}$ conforming to $\Phi$ satisfies $\col(A[X_1,E-(Y_0 \cup Y_1)]) \subseteq \col(A[X_1,Y_1])$, and so $\rank(A[X_1,E-(Y_0 \cup Y_1)]) < \rank(A[X_1,E-Y_1])$, which gives $r(M(A) \del (Y_1 \cup Y_0)) < r(M(A) \del Y_1)$. Therefore $\lambda_{M(A) \del Y_1}(Y_0) < r_{M(A)}(Y_0) \le |Y_0|$. If $r(M(A) \del Y_1) > |Y_0|$ then it follows that $M(A) \del Y_1$ is not vertically $(|Y_0| + 1)$-connected, so (\ref{m0}) holds, a contradiction.
	\end{proof}

	Recall that the rows of $A_1[X_1,Y_0 \cup Y_1]$ are linearly independent. We may therefore assume by the first claim, and applying  row-operations to $A_1[X_1]$, that there exists $T \subseteq Y_1$ so that $A_1[X_1,T] = -J$  for some bijection matrix $J$. Recall that $|B_0| > |X_1|$; let $X_1' \subseteq B_0$ be a set whose elements we associate with those in $X_1$, and let $J' \in \bF^{X_1' \times T}$ be a copy of $J$. Let $r \in B_0 - X_1'$. Consider the matrix $M = M(A)$, where 
	\begin{center}
		$A= $\begin{tabular}{c|c|c|c|c|}
			\tcol{} & \tcol{$F$} & \tcol{$T$} &\tcol{$Y_1-T$} & \tcol{$Y_0$}\\
			\cline{2-5}
			$X_1$ & $0$ & $-J$ & \multirow{2}{*}{$A_1[X,Y_1-T]$} & \multirow{2}{*}{$A_1[Y_0]$}\\
			\cline{2-3}
			$X_0$ & $P_0$ & $A_1[X_0,T]$ &  &   \\ 
			\cline{2-5}
			$X_1'$ & \multirow{3}{*}{$Q$} & $J'$ & \multirow{2}{*}{$0$} & \multirow{3}{*}{$0$} \\
			\cline{3-3}
			$B_0-(X_1' \cup \{r\})$ & & $0$ & &  \\
			\cline{3-4}
			$r$ & & $0$ & $\one{Y_1-T}$ &  \\
			\cline{2-5}
		\end{tabular}.
	\end{center}
	The matroid $M(A)$ is isomorphic to a matroid conforming to $\Phi$, as $A$ is obtained from a certain matrix conforming to $\Phi$ (in which $Z$ is a copy of $Y_1$) by removing the columns in $Y_1$ and then renaming the set $Z$ as $Y_1$. Thus $M \in \cM(\Phi)$. 
	
	Let $M' = M \con T$. Since $J'$ is a copy of $J$, we have $M = M(A')$ where	
	\begin{center}
		$A'= $\begin{tabular}{c|c|c|c|}
			\tcol{} & \tcol{$F$} &\tcol{$Y_1-T$} & \tcol{$Y_0$}\\
			\cline{2-4}
			$X_0$ & $P_0$  & $H_1$ & $H_2$   \\ 
			\cline{2-4}
			$X_1'$ & \multirow{3}{*}{$Q$}  & $A_1[X_1,Y_1-T]$ & $A_1[X_1,Y_0]$  \\
			\cline{3-4}
			$B_0-(X_1' \cup \{r\})$ & & $0$ & \multirow{2}{*}{$0$}  \\
			\cline{3-3}
			$r$ &  & $\one{Y_1-T}$ &  \\
			\cline{2-4}
		\end{tabular},
	\end{center}
	where the sets $X_1'$ and $X_1$ are identified, and $H_1,H_2$ are some matrices. Now $M' = M \con T \in \ol{\cM{(\Phi)}}$, but also $M' = M(A')$ is evidently a rank-$(r(N))$ extension of the matroid $N = M\sqbinom{P_0}{Q}$. By the choice of $N$, we may assume that $M'$ contains no simple extension, so in fact $A'[B_0]$ is a $\Gamma$-frame matrix up to column scalings. Thus every column of $A_1[X_1,Y_1-T]$ is either a zero vector or a weight-$1$ vector whose nonzero entry is in $-\Gamma$. The same is true of $A_1[X_1,T] = -J$.
	
	
	Consider an arbitrary matroid $M \in \cM(\Phi)$, so $M = M(A) \del Y_1$ for some matrix $A \in \bF^{B \times E}$ conforming to $\Phi$. Every column of $A[X_1,Y_0]$ is parallel to a column of a $\Gamma$-frame matrix, so the same is true of $A[X_1 \cup B,Y_0]$. The columns of $A[B,Z]$ are unit vectors and the columns of $A[X_1,Z]$ are columns of $A[X_1,Y_1]$ so each is either a zero vector or a weight-$1$ vector with nonzero entry in $-\Gamma$; thus, each column of $A[X_1 \cup B,Z]$ is parallel to a column of a $\Gamma$-frame matrix. The same is evidently true of $A[X_1 \cup B,E-(Y_0 \cup Y_1 \cup Z)]$; thus, $A[X_1 \cup B,E-Y_1]$ is a $\Gamma$-frame matrix up to column scalings. Since $|X_0| = t$ it follows that $M(A) \del Y_1 \in \cG(\Gamma)^t$ and so (\ref{m1}) holds. 
	\end{proof}

\section{Prime Geometries}
	
	We now use Theorem \ref{templatetech} to prove a general result that will easily imply Theorems~\ref{maintwo}, ~\ref{mainthree} and ~\ref{mainodd}, subject to Hypothesis \ref{structure1}.
	We emphasize that this is the only section of the paper that assumes Hypothesis \ref{structure1}.
	Note that this result implies Theorem~\ref{simplifiedmain} because $\xpinch{t}$ and $\sqpinch{t}$ both contain simple extensions of $\DG(n,\Gamma)^t$ for  all $n \ge t$.
	

	\begin{theorem}\label{bigmain}
		Suppose Hypothesis \ref{structure1} holds.
		Let $\bF = \GF(p)$ be a prime field and $\cM$ be a quadratic minor-closed class of $\bF$-represented matroids. There is a subgroup $\Gamma$ of $\bF^*$ and some $t,\alpha \in \nni$ so that $\cG(\Gamma)^t \subseteq \cM$ and 
		\[f_{p,|\Gamma|,t}(n) \le h_{\cM}(n) \le f_{p,|\Gamma|,t}(n) + \alpha n\]
		 for all sufficiently large $n$. Moreover, either
		 \begin{enumerate}[(a)]
		 	\item\label{ma0} $\alpha = 0$ and every extremal matroid $N$ in $\cM$ of sufficiently large rank is isomorphic to $\DG(r(N),\Gamma)^t$, 
			\item\label{ma1} $\Gamma = \bF^*$ and $\sqpinch{t} \subseteq \cM$, or
			\item\label{ma2} $\Gamma \ne \bF^*$ and $\xpinch{t} \subseteq \cM$ for some $x \in \bF^* - \Gamma$. 
		\end{enumerate}
	\end{theorem}
	\begin{proof}		
		For each integer $s$, let $\cM_s$ denote the class of vertically $s$-connected matroids in $\cM$ of rank at least $s$. Since $\cM$ is a quadratic class, it contains all graphic matroids, so $\cM_s \ne \varnothing$ for all $s \in\nni$. 
		By Hypothesis~\ref{structure1} and Corollary~\ref{strongstructure} there are finite sets $\bT$ and $\bT^*$ of reduced frame templates and integers $k$ and $m$ such that each simple matroid $M \in \cM_k$ with an $M(K_m)$-minor is either in $\cM(\Phi)$ for some $\Phi \in \bT$ or in $\cM^*(\Psi)$ for some $\Psi \in \bT^*$, while $\cM$ contains $\cM(\Phi)$ for all $\Phi \in \bT$ and $\cM^*(\Psi)$ for all $\Psi \in \bT^*$. 
Since $\cM_{\max(k,m)} \ne \varnothing$ we know that $\bT \cup \bT^* \ne \varnothing$. 
By Theorem \ref{grt}, there is an integer $\alpha_m$ so that each matroid $M$ with no $M(K_m)$-minor satisfies $\elem(M)\le \alpha_m\cdot r(M)$.

		Note that the class $\cG(\{1\})^0$ of graphic matroids is contained in $\cM$. If $\cG(\{1\})^t \subseteq \cM$ for all $t \in \nni$ then $\cM$ contains all projective geometries over $\bF$ and is thus not a quadratic class; let $t \in \nni$ and $\Gamma \le \bF^*$ be such that $\cG(\Gamma)^t \subseteq \cM$ and $p^t|\Gamma|$ is as large as possible. Let $c = \max_{\Phi \in \bT \cup \bT^*}c(\Phi)$. Note that the function $f_{p,g,t'}(x)$ is quadratic in $x$ with leading term $\tfrac{1}{2}gp^tx^2$; let $n_0$ be an integer for which every $x \ge n_0$ satisfies $\max(p^c(3x + 6c + 1), \alpha_mx) < f_{p,|\Gamma|,t}(x)-x$, and let $n_1$ be an integer so that every $x \ge n_1$ satisfies $f_{p,g,t'}(x) + 2p^{t'+1}cx < f_{p,|\Gamma|,t}(x)-x$ for all $t',g \in \nni$ such that $p^{t'}g < p^t|\Gamma|$.
		Let $k_0 = \max(t,k,c+1,n_0,n_1)$.
		
		\begin{claim}
			If $M \in \cM_{k_0}$ is simple and $\elem(M) \ge f_{p,|\Gamma|,t}(r(M))-r(M)$, then $M$ conforms to a template $\Phi' = (\Gamma',C',X',Y_0',Y_1',A_1',\Delta',\Lambda') \in \bT$ for which $\Gamma' = \Gamma$ and $\dim(\Lambda') = t$. 
		\end{claim}
		\begin{proof}[Proof of claim:]
			Let $r = r(M)$. Since $M \in \cM_{k_0}$ and $k_0\ge n_0$, the matroid $M$ has an $M(K_m)$-minor. Then since $M \in \cM_{k_0}$ and $k_0 \ge k$, we have $M \in \cM^*(\Psi)$ for some $\Psi \in \bT^*$ or $M \in \cM(\Phi)$ for some $\Phi \in \bT$. In the former case, since $r \ge k_0 \ge n_0$ we have  $|M| < p^c(3r + 6c + 1) < f_{p,|\Gamma|,t}(r)-r$ by Lemma~\ref{dualdensity}, a contradiction. Therefore $M \in \cM(\Phi')$ for some $\Phi' \in \bT$; let $\Phi' = (\Gamma',C',X',Y_0',Y_1',A_1',\Delta',\Lambda')$ and $t' = \dim(\Lambda')$. 
			
			Lemma~\ref{subclass} implies that $\cG(\Gamma')^{t'} \subseteq \ol{\cM(\Phi')} \subseteq \cM$, so $p^{t'}|\Gamma'| \le p^t|\Gamma|$ by our choice of $\Gamma$ and $t$. Using Lemma~\ref{templatedensity}, we have 
			\[f_{p,|\Gamma|,t}(r)-r < |M| < f_{p,|\Gamma'|,t'}(r) + p^{t'+1}c(r+c) < f_{p',|\Gamma'|,t}(r) + 2p^{t'+1}cr .\]
			If $p^{t'}|\Gamma'| < p^t|\Gamma|$ then the above and the fact that $r \ge k_0 \ge n_1$ give a contradiction. Thus $p^{t'}|\Gamma'| = p^t|\Gamma|$, which implies that $t' = t$ and $|\Gamma'| = |\Gamma|$; since $\bF^p$ is cyclic this gives $\Gamma' = \Gamma$. 		
		\end{proof}

		Since $\cG(\Gamma)^t \subseteq \cM$, we have $h_{\cM}(n) \ge f_{p,|\Gamma|,t}(n)$ for all $n \ge t$. Let $d \in \{-1,0,1, \dotsc, 2p^{t+1}c\}$ be maximal such that $h_{\cM}(n) > f_{p,|\Gamma|,t}(n) + dn$ for infinitely many $n$. By Theorem~\ref{connreduction}, applied with $f(x) = f_{p,|\Gamma|,t}(x) + dx$, there is a simple rank-$r$ matroid $M \in \cM_{k_0}$ for which $r\ge k_0$ and $|M| > f_{p,|\Gamma|,t}(r) + d r$. 
		
		By the claim, there is a template $\Phi' = (\Gamma,C',X',Y_0',Y_1',A_1',\Delta',\Lambda') \in \bT$ for which $M \in \cM(\Phi)$ while $\dim(\Lambda) = t$. By Lemma~\ref{templatedensity} we have $|M| < f_{p,|\Gamma|,t}(r) + 2p^{t+1}cr$, from which we see that $d < 2p^{t+1}c$; by the maximality of $d$ this gives the required upper bound on $h_{\cM}(n)$ with $\alpha = 2p^{t+1}c$. 
				
		We now apply Theorem~\ref{templatetech} to $\Phi'$. Since $M \in \cM(\Phi')$ is vertically $(c+1)$-connected and has rank at least $c+1$, we know that \ref{templatetech}(\ref{m0}) does not hold. Outcomes \ref{templatetech}(\ref{m2}) and \ref{templatetech}(\ref{m3}) imply (\ref{ma1}) and (\ref{ma2}) respectively, so we may assume that  \ref{templatetech}(\ref{m1}) holds. Thus $M \in \cG(\Gamma)^t$ and so $\elem(M) \le f_{p,|\Gamma|,t}(r)$; if $d \ne -1$ this is a contradiction, so $d = -1$. By the choice of $d$ it follows that $h_{\cM}(n) = f_{p,|\Gamma|,t}(n)$ for all sufficiently large $n$. 
		
		By Lemma~\ref{equalityhc}, there is some $k_1 \in \nni$ such that every simple $N \in \cM$ with $|N| = f_{p,|\Gamma|,t}(r(N))$ and $r(N) \ge k_1$ is in $\cM_{k_0}$. Consider such an $N$; by the claim there is a template $\Phi'' = (\Gamma,C'',X'',Y_0'',Y_1'',A_1'',\Delta'',\Lambda'')$ in $\bT$ with $\dim(\Lambda'') = t$ and $N \in \cM(\Phi'')$. Again, we may assume that outcomes (\ref{m0}),(\ref{m2}),(\ref{m3}) of Theorem~\ref{templatetech} do not hold for $\Phi''$ and so (\ref{m1}) does; thus $N \in \cG(\Gamma)^t$. Since $\DG(r(N),\Gamma)^t$ is the unique simple rank-$r(N)$ matroid with $f_{p,|\Gamma|,t}(n) = |N|$ elements, we must have $N \cong \DG(r(N),\Gamma)^t$ as required.  
	\end{proof}


	Finally, we prove the corollary that will give the extremal function for excluding geometries. Note that the class $\cM$ in the theorem statement is quadratic as it contains the class $\cG(\{1\})^0$ of graphic matroids. 

	\begin{theorem}\label{excludeN}
		Suppose Hypothesis \ref{structure1} holds. Let $\bF = \GF(p)$ be a prime field and $N$ be a nongraphic $\bF$-represented matroid. Let $\cM$ be the class of $\bF$-represented matroids with no $N$-minor. Let $\Gamma \le \bF^*$ and $t \in \nni$ be such that $p^t|\Gamma|$ is as large as possible subject to $N \notin \cG(\Gamma)^t$. If \begin{itemize}
			\item $\Gamma = \bF^*$ and $N \in \sqpinch{t}$ or
			\item $\Gamma \ne \bF^*$ and $N \in \xpinch{t}$ for all $x \in \bF^* - \Gamma$, 
		\end{itemize}
		then for all sufficiently large $n$ we have $h_{\cM}(n) = f_{p,|\Gamma|,t}(n)$ and every rank-$n$ extremal matroid in $\cM$ is isomorphic to $\DG(n,\Gamma)^t$. 
	\end{theorem}
	\begin{proof}
		Clearly $\cG(\Gamma)^t \subseteq \cM$. Let $\Gamma' \le \bF^*$ and $t',\alpha \in \nni$ be given by Theorem~\ref{bigmain} for $\cM$. If $|\Gamma'|p^{t'} < |\Gamma|p^t$ then, since $f_{p,|\Gamma'|,t'}(n) \approx \tfrac{1}{2}p^{t'}|\Gamma'|n^2$, we have $h_{\cM}(n) \le f_{p,|\Gamma'|,t'}(n) + \alpha n < f_{p,|\Gamma|,t}(n) \le h_{\cM}(n)$ for all large $n$, a contradiction. So $|\Gamma'|p^{t'} = |\Gamma|p^t$; since $|\Gamma| < p$ and $\bF^*$ is cyclic this gives $\Gamma = \Gamma'$ and $t = t'$. By hypothesis, we see that \ref{bigmain}(\ref{ma1}) and \ref{bigmain}(\ref{ma2}) cannot hold for $\cM$, so we have \ref{bigmain}(\ref{ma0}) which implies the result. 
	\end{proof}

	Theorems~\ref{maintwo}, ~\ref{mainthree} and ~\ref{mainodd} follow from the above result and Lemmas~\ref{techtwo}, ~\ref{techthree} and ~\ref{techodd}, respectively.

	\section*{References}
	\newcounter{refs}
	\begin{footnotesize}
	\begin{list}{[\arabic{refs}]}
		{\usecounter{refs}\setlength{\leftmargin}{10mm}\setlength{\itemsep}{0mm}}
		
		\item\label{cgovz}
		B. Clark, K. Grace, J. Oxley, S.H.M. van Zwam,
		On the highly connected dyadic, near-regular, and sixth-root-of-unity matroids, 
		arXiv:1903.04910 [math.CO].
		
		\item\label{ggwstructure}
		J. Geelen, B. Gerards, G. Whittle, 
		The highly connected matroids in minor-closed classes, 
		Ann. Comb. 19 (2015), 107--123.
		
		\item \label{gkw09}
		J. Geelen, J.P.S. Kung, G. Whittle, 
		Growth rates of minor-closed classes of matroids, 
		J. Combin. Theory. Ser. B 99 (2009) 420--427.
		
		\item\label{grace}
		K. Grace, 
		The templates for some classes of quaternary matroids, 
		J. Combin. Theory. Ser. B 146 (2021) 286--363.
				
		\item\label{gvz}
		K. Grace, S.H.M. van Zwam, 
		Templates for binary matroids, 
		SIAM J. Discrete Math. 31 (2017), 254--282.	
		
		\item\label{gvzcounterexample}
		K. Grace, S.H.M. van Zwam, 
		On perturbations of highly connected dyadic matroids, 
		Ann. of Comb. 22 (2018), 513--542.
				
		\item\label{heller}
		I. Heller, 
		On linear systems with integral valued solutions,
		Pacific. J. Math. 7 (1957) 1351--1364.
		
		\item\label{kmpr}
		J.P.S. Kung, D. Mayhew, I. Pivotto, G.F. Royle,
		Maximum size binary matroids with no $\AG(3,2)$-minor are graphic,
		SIAM J. Discrete Math. 28 (2014), 1559--1577.
		
		\item\label{gn}
		J. Geelen, P. Nelson, 
		Matroids denser than a clique, 
		J. Combin. theory. Ser. B 114 (2015), 51--69.
		
		\item \label{oxley}
		J. G. Oxley, 
		Matroid Theory (2nd edition),
		Oxford University Press, New York, 2011.
		
		\item\label{walsh}
		Z. Walsh, 
		On the density of binary matroids without a given minor, 
		MMath Thesis, University of Waterloo, 2016. 
		
		\item\label{zaslav}
		T. Zaslavsky, Signed graphs, 
		Discrete Appl. Math. 4 (1982) 47--74
		
	\end{list}	
\end{footnotesize}

\end{document}